\documentclass[11pt]{article}
\usepackage{amsthm,amsmath,amssymb}
\usepackage[left=1.00in,bottom=1.25in]{geometry}
\usepackage{graphicx}
\usepackage{gensymb}
\usepackage{float}
\usepackage{url}
\usepackage{titling}
\usepackage{nicematrix}
\usepackage{comment}

\usepackage{titlesec}
\usepackage[labelformat=simple]{subcaption}

\usepackage{adjustbox}
\usepackage{tikz}
\usepackage{wrapfig}
\usepackage{dsfont}
\usepackage{color}
\usepackage{xcolor}
\usepackage{bm}
\usepackage{setspace}
\usepackage{array,multirow}
\usepackage{enumerate}

\titleformat{\subsection}[runin]
  {\normalfont\large\bfseries}{\thesubsection}{1em}{}
\titlelabel{\thetitle.\quad}

\titleformat{\section}[runin]
  {\normalfont\large\bfseries}{\thesection}{1em}{}

\newtheorem*{remark}{Remark}
\newtheorem{theorem}{Theorem}[section]
\newtheorem{lemma}[theorem]{Lemma}

\numberwithin{equation}{section}

\title{Well-posedness of a novel Lagrange multiplier formulation for fluid-poroelastic interaction}
\author{Amy de Castro
\thanks{Department of Mathematics, University of Utah, Salt Lake City, UT 84123
    ({\tt amy.de.castro@utah.edu}). Partially supported by the NSF under grant number DMS-2207971.}
\and
Hyesuk Lee
  \thanks{School of Mathematical and Statistical Sciences, Clemson
        University, Clemson, SC  29634-0975 ({\tt hklee@clemson.edu}).
    Partially supported by the NSF under grant number DMS-2207971.}
}
\date{}

\begin{document}

\maketitle
\vspace{-10mm}
\begin{abstract}
 We introduce a novel monolithic formulation that employs Lagrange multipliers (LMs) to couple a fluid flow governed by the time-dependent Stokes equations with a poroelastic structure described by the Biot equations. The formulation is developed in detail, and we establish the well-posedness of both the semi-discrete and fully discrete saddle point problems. We further prove the stability of the fully discrete system. This saddle point formulation, which utilizes three LMs, is designed to enable a partitioned approach that completely decouples the Stokes and Biot subdomains, and this approach will be explored in a subsequent work.
\end{abstract}

\section{Background and Research Goals}
In the context of a linear elastic structure, the structural deformation is the main process at work. Poroelasticity, on the other hand, describes fluid flow through a deformable porous medium, typically governed by Darcy’s law. 
Poroelastic materials are modeled by the Biot equations, which capture the interaction between the deformation of an elastic structural skeleton and the motion of a saturating fluid. 
These models have important applications in geoscience, such as groundwater flow and flow through fractured rock formations  \cite{Detournay_1993, Murad_2001}, as well as in biomedical fields including arterial and vascular blood flow and drug delivery processes transport \cite{Banks_2017, Bociu_2021, Calo_2008, Causin_2014}. 

Analysis and numerical methods for the Biot model can be categorized broadly by how many variables they include and the dynamics under consideration. The most widely used formulations are two- or three-field models, although more complex four- or five-field models have also been proposed \cite{Caucao_2022, Kumar_2020}. In the two-field model, the primary variables are structural deformation and pore pressure, while the three-field model includes the fluid velocity. 
Depending on the application, the model may account for structural acceleration to capture fully dynamic behavior, or neglect it in quasistatic approximations. In this work, we adopt the fully dynamic two-field Biot model for coupling,
but provide a brief overview of existing methods and analysis for a range of fluid-poroelastic structure interaction (FPSI) models.

A semigroup approach has been used to analyze existence of strong and weak solutions for the fully inertial two-field Biot system coupled with incompressible Stokes equations \cite{Avalos_2024} and the three-field Biot model coupled with the compressible Stokes equations \cite{Showalter_2005}. 
Other analytic results for FPSI systems include well-posedness for a nonlinear coupling between a non-Newtonian fluid and the quasi-static Biot model \cite{Ambartsumyan_2019}.
The existence of weak solutions for the incompressible Navier-Stokes equations with a two-field Biot model under a small data assumption is proven in \cite{Cesmelioglu_2017}, although the small data assumption is only required due to the nonlinearity of the Navier-Stokes equations. 
Motivated by biological applications such as vascular stents, Stokes flow may be coupled to a multilayered Biot model comprised of a thin poroelastic plate and a thick Biot layer \cite{Bociu_2021}. Well-posedness results have been shown for the linear, fully dynamic scenario and for the nonlinear, quasistatic model through Rothe's method, solving a series of elliptic, semidiscrete in time problems.

We consider the coupling of the fully dynamic two-field Biot equations with the dynamic Stokes equations, using three LMs to enforce interface conditions between the subdomains. 
Instead of reformulating the Biot equations as a first-order system in time by introducing structural velocity as an additional variable, we retain the original second order in time formulation. 
Our goal is to formulate the coupled problem as a saddle point system in a form suitable for domain decomposition, providing a foundation for future work to develop a non-iterative partitioned method based on this formulation.
While some existing FPSI formulations that use LMs introduce only a single multiplier to represent the pore pressure on the interface 
\cite{Ambartsumyan_2019, Ambartsumyan_2018}, and others include additional multipliers for structure or fluid velocity 
\cite{Li_2022_Hydro, Caucao_2022},  our approach defines the multipliers differently. 
As our LMs are defined differently, they result in a unique saddle point formulation which will suggest a partitioned technique for the solution of the coupled system.

A study has been reported on reducing the size of the discrete system of the monolithic formulation using LMs within the framework of FPSI.
In \cite{Caucao_2022}, three LMs representing restrictions of the fluid velocity, structure velocity, and pore pressure to the interface are utilized. 
The formulation employs a stress–velocity–vorticity approach for the Stokes equations, a velocity–pressure formulation for Darcy flow, and a stress–displacement–rotation model for elasticity. 
A vertex quadrature rule is presented which allows for a reduction of the size of the monolithic system to be solved - from eleven variables down to six, 
as five of the variables' DOFs can be decoupled from the rest by virtue of the quadrature rule. 
The reduced matrix formulation is solved as a monolithic system, with the Darcy pressure, structural and Stokes velocities, and three LMs remaining. This decoupling, however, is not of the Stokes and Biot subdomains themselves, as one must still solve a system involving variables from both subdomains.

The formulation we present of the FPSI system  will allow for the decoupling of the physical subdomain systems through the solution of an interface problem. We focus on deriving and proving well-posedness and stability of the formulation, with the aim of developing a foundation for a truly partitioned method in which the Stokes and Biot subdomains may be entirely decoupled from each other. 
The use of LMs leads to a saddle point problem that facilitates domain decomposition along the physical interface, allowing parallel computation of local problems. This contrasts with earlier approaches \cite{Ambartsumyan_2018, Li_2022_Hydro, Wang_2025} which address a fully monolithic system coupling variables from both subdomains or solve local problems sequentially.

The paper is outlined as follows. We discuss the model equations in Section \ref{FPSI_FEM:sec:Continuous Model-3LMs} and present our weak form and saddle point structure, leading to a well-posed semi-discrete formulation. Turning to the fully discrete model, we demonstrate well-posedness and give stability results in Sections \ref{FPSI_FEM:sec:FullyDiscreteModelUpdates} and \ref{FPSI_FEM:sec:stability}. Lastly, we offer conclusions and expound on our future work of developing a partitioned method based on this formulation in Section \ref{FPSI_joint:sec:Conclusion}. 




	\section{Model Equations and Semi-Discrete Model }\label{FPSI_FEM:sec:Continuous Model-3LMs}
We begin by developing the weak form and semi-discrete monolithic formulation of the Stokes-Biot system. Consider a Lipschitz domain $\Omega$ which is divided into two open regions: $\Omega_f \in \mathbb{R}^d$ containing the fluid, and $\Omega_p \in \mathbb{R}^d$ containing the poroelastic structure, for $d=2,3$. We assume these domains are non-overlapping and share an interface $\gamma$. 
As the problem is time dependent, we take $T>0$ to be a given final time.

The flow in the fluid domain is modeled by the transient Stokes equations, and the two-field Biot equations describe the poroelastic material in $\Omega_p$. The resulting unknown functions are the fluid velocity $\bm{u}(\bm{x},t)$, the fluid pressure $p_f(\bm{x},t)$, the structural displacement $\bm{\eta}(\bm{x},t)$, and the pore pressure $p_p(\bm{x},t)$. With given body forces $\bm{f}_f, \bm{f}_\eta,$ and source or sink $f_p$, the model problem reads:

\noindent \textit{Find  $\bm{u} \in \Omega_f \times (0,T] \mapsto \mathbb{R}^d, \hspace{1mm} p_f \in \Omega_f \times (0,T] \mapsto \mathbb{R}, \hspace{1mm} \bm{\eta} \in \Omega_p \times (0,T] \mapsto \mathbb{R}^d, p_p \in \Omega_p \times (0,T] \mapsto \mathbb{R} $ s.t.} 
\begin{align}
\rho_f  \frac{\partial \bm{u}}{\partial t} - 2 \nu_f \nabla \cdot D(\bm{u}) + \nabla p_f &= \bm{f}_f \hspace{5mm} \text{ in }\Omega_f \times (0,T] \label{FPSI_FEM:StokesMom} \\
\nabla \cdot \bm{u} &= 0 \hspace{8mm} \text{ in }\Omega_f \times (0,T]  \label{FPSI_FEM:StokesMass} \\
\rho_p \frac{\partial^2 \bm{\eta}}{\partial t^2} - 2\nu_p \nabla \cdot D(\bm{\eta}) - \lambda \nabla (\nabla \cdot \bm{\eta}) + \alpha \nabla p_p &= \bm{f}_\eta \hspace{6mm} \text{ in }\Omega_p \times (0,T] \label{FPSI_FEM:Structure1} \\
s_0 \frac{\partial p_p}{\partial t} + \alpha \nabla \cdot \frac{\partial \eta}{\partial t} - \nabla \cdot \kappa \nabla p_p &= f_p \hspace{7mm} \text{ in } \Omega_p \times (0,T]. \label{FPSI_FEM:Structure2}
\end{align}
 Above, $D(\cdot)$ is the deformation rate tensor, defined as $D(\bm{v}) := \frac{1}{2}(\nabla \bm{v} + (\nabla \bm{v})^T)$. Densities are denoted by $\rho_f, \rho_p$, fluid viscosity by $\nu_f$, Lam\'e parameters by $\nu_p, \lambda$, and the Biot-Willis constant by $\alpha$. The constrained specific storage coefficient is denoted by $s_0$, and $\kappa$ represents the hydraulic conductivity. Although in general $\kappa$ is a tensor, we simplify here by considering an isotropic porous material so that $\kappa$ becomes a scalar. Each parameter is assumed to be constant in time for our analysis.  

Initial conditions are provided for $\bm{u}, \bm{\eta}$, and $p_p$. With stress tensors $\bm{\sigma}_f := 2\nu_f D(\bm{u}) - p_f I$ and $\bm{\sigma}_p := 2\nu_p D(\bm{\eta}) + \lambda (\nabla \cdot \bm{\eta})I - \alpha p_p I$, boundary data is given as
\begin{align}
\begin{split}\label{FPSI_FEM:noninterBCs}
\bm{\sigma}_f  \bm{n}_f &= \bm{u}_N \hspace{5mm} \text{ on } \Gamma_N^f  \times (0,T], \hspace{8mm}
\bm{u} = \bm{0} \hspace{5mm} \text{ on } \Gamma_D^f \times (0,T] \\
\bm{\sigma}_p  \bm{n}_p &= \bm{\eta}_N \hspace{5mm} \text{ on } \widetilde{\Gamma}_N^p \times (0,T], \hspace{8mm}
\bm{\eta} = \bm{0} \hspace{5mm} \text{ on } \widetilde{\Gamma}_D^p \times (0,T] \\
(\kappa \nabla p_p) \cdot \bm{n}_p &= p_N \hspace{5mm} \text{ on } \Gamma_N^{p} \times (0,T], \hspace{8mm} p_p = 0 \hspace{5mm} \text{ on } \Gamma_D^p \times (0,T],
\end{split}
\end{align}
where $\Gamma^f = \Gamma^f_N \cup \Gamma^f_D \cup \gamma$ is the Lipschitz continuous boundary of $\Omega_f$. Likewise, the boundary $\Gamma^p$ of $\Omega_p$ may be written as $\Gamma^p = \Gamma^p_N \cup \Gamma^p_D \cup \gamma$ = $\widetilde{\Gamma}^p_N \cup \widetilde{\Gamma}^p_D \cup \gamma$. We employ two notations in order to allow different types of boundary conditions to be defined for the displacement and pore pressure along the same spatially coincident portion of $\Gamma^p$. In each domain $\Omega_k$, $k \in \{f,p\}$, we assume the measure of $\Gamma^k_N$ and $\Gamma^k_D$ are nonzero and take the unit vectors $\bm{n}_k$ to be outward normal to the domains. The unit vector $\bm{\tau}_\gamma$ is tangential to the interface $\gamma$.

\noindent We complete the system by providing the following interface conditions representing mass conservation, balance of stresses, and the Beavers-Joseph-Saffman (BJS) condition, where $\beta$ is the resistance parameter in the tangential direction:
\begin{align}
\bm{u} \cdot \bm{n}_f &= -\left(\frac{\partial \bm{\eta}}{\partial t} - \kappa \nabla p_p\right) \cdot \bm{n}_p \hspace{35mm} \text{ on } \gamma \times (0,T] \label{FPSI_FEM:Inter1} \\
\bm{\sigma}_f \bm{n}_f &= -\bm{\sigma}_p \bm{n}_p \hspace{57mm} \text{ on } \gamma \times (0,T] \label{FPSI_FEM:Inter2} \\
\bm{n}_f\cdot \bm{\sigma}_f \bm{n}_f &= -p_p \hspace{62mm} \text{ on } \gamma \times (0,T] \label{FPSI_FEM:Inter3} \\
\bm{n}_f\cdot \bm{\sigma}_f \bm{\tau}_\gamma^\ell &= -\beta \left(\bm{u} - \frac{\partial \bm{\eta}}{\partial t}\right) \cdot \bm{\tau}_\gamma^\ell \hspace{5mm} \text{ for } 1 \leq \ell \leq d-1, \hspace{5mm} \text{ on } \gamma \times (0,T]. \label{FPSI_FEM:Inter4}
\end{align}

\subsection{Derivation of Weak Form}
Define the following continuous spaces:
\begin{align}\label{FPSI_FEM:ContSpaces}
\begin{split}
&U :=\{\bm{v} \in \bm{H}^1(\Omega_f): \bm{v} = \bm{0} \text{ on } \Gamma_D^f\}, \hspace{5mm}
X := \{\bm{\varphi} \in \bm{H}^1(\Omega_p): \bm{\varphi} = \bm{0} \text{ on } \widetilde{\Gamma}_D^p\} \\
&Q_f :=L^2(\Omega_f), \hspace{38mm} Q_p := \{q_p \in H^1(\Omega_p) : q_p = 0 \text{ on } \Gamma_D^p\}.
\end{split}
\end{align}

Throughout this paper, we use boldface font to distinguish a vector-valued function, such as $\bm{u},\bm{\eta}$, from a scalar-valued function such as $p_p, p_f$. Likewise, function spaces will be typeset in bold to indicate their correspondence to a vector-valued function.
Let $H^s(\Omega_i)$ be the Hilbert space of order $s$ defined on subdomain $\Omega_i$, $i \in \{ f,p\}$. The notation $(\cdot,\cdot) = (\cdot,\cdot)_{\Omega_i}$ represents the $L^2$ inner products on $\Omega_i$. We may drop the subscript $\Omega_i$ from the inner product or norm notation if it is clear from context. A duality product between $H^s$ and its dual space for $s>0$ is denoted by $\langle \cdot,\cdot \rangle_{\Omega_i}$.  We define the $\bm{H}^1$ norm for vector-valued functions $\bm{v} \in \bm{H}^1(\Omega_i)$ as $||\bm{v}||_{1,\Omega_i}^2 := ||\bm{v}||_{0,\Omega_i}^2 + ||D(\bm{v})||_{0,\Omega_i}^2$, with  $||w||_{1,\Omega_i}^2 := ||w||_{0,\Omega_i}^2 + ||\nabla w||_{0,\Omega_i}^2$ the corresponding norm for scalar-valued functions $w \in H^1(\Omega_i)$. Likewise along a portion of the boundary $\Gamma$, we take $(\cdot,\cdot)_\Gamma$ to be the $L^2$ inner product, and $\langle \cdot,\cdot \rangle_{\Gamma}$ to represent a dual product.

In $\mathbb{R}^2$, we introduce three Lagrange multipliers (LMs) for quantities on the interface, $g_1 \in \Lambda_{g1} := H^{-1/2}(\gamma)$, $g_2 \in \Lambda_{g2} := L^2(\gamma)$, and $\lambda_p \in \Lambda_\lambda := H^{1/2}(\gamma)$, defined on $\gamma \times (0,T]$ by

\begin{equation*}
g_1 := (\bm{\sigma}_f \bm{n}_f)\cdot \bm{n}_f, \quad g_2 := (\bm{\sigma}_f \bm{n}_f)\cdot \bm{\tau}_\gamma, \quad \text{ and }  \lambda_p := \kappa \nabla p_p \cdot \bm{n}_p.
\end{equation*}
Treating the normal and tangential components of $\bm{\sigma}_f \bm{n}_f$ as independent LMs will prove to be a judicious choice for later well-posedness proofs, as $L^2$ regularity is only required of the tangential component instead of the entire normal stress. The higher regularity for $\lambda_p$ is needed for well-posedness.

Recall that we may write a general vector valued function as a sum of its normal and tangential components. With this in mind, the vector $(\bm{\sigma}_f \bm{n}_f)|_\gamma$ may be rewritten as \begin{equation*}
    \bm{\sigma}_f \bm{n}_f\Big|_{\gamma} = \Big((\bm{\sigma}_f \bm{n}_f)\cdot \bm{n}_f\Big)\Big|_{\gamma} \bm{n}_f + \Big((\bm{\sigma}_f \bm{n}_f)\cdot \bm{\tau}_\gamma\Big)\Big|_\gamma \bm{\tau}_\gamma = g_1 \bm{n}_f + g_2 \bm{\tau}_\gamma,
\end{equation*}
Multiplying by appropriate test functions and integrating by parts, we derive the weak form of \eqref{FPSI_FEM:StokesMom}-\eqref{FPSI_FEM:Structure2}, for given $\bm{f}_f \in (\bm{H}^1(\Omega_f))^*$, $\bm{f}_\eta \in (\bm{H}^1(\Omega_p))^*$, and $f_p \in (H^1(\Omega_p))^*$:

\noindent \textit{Find $\bm{u} \in U, p_f \in Q_f, \bm{\eta} \in X, p_p \in Q_p, g_1 \in \Lambda_{g1}, g_2 \in \Lambda_{g2},$ and $ \lambda_p \in \Lambda_\lambda$ such that for a.e. $t \in (0,T]$, }
\begin{align}
\begin{split} \label{FPSI_FEM:WF:subdomains}
&\rho_f \Big( \frac{\partial \bm{u}}{\partial t}, \bm{v} \Big)_{\Omega_f} + 2 \nu_f \left( D(\bm{u}), D(\bm{v}) \right)_{\Omega_f} - (p_f, \nabla \cdot \bm{v})_{\Omega_f} - \langle g_1 \bm{n}_f,\bm{v}\rangle_\gamma  - ( g_2 \bm{\tau}_\gamma,\bm{v})_\gamma \\
&\hspace{15mm}= \langle\bm{f}_f,\bm{v}\rangle_{\Omega_f} + \langle\bm{u}_N, \bm{v}\rangle_{\Gamma_N^f} \hspace{3mm} \forall \ \bm{v} \in U, \\
&( \nabla \cdot \bm{u}, q )_{\Omega_f} = 0 \hspace{3mm} \forall \ q \in Q_f, \\
&\rho_p \Big( \frac{\partial^2 \bm{\eta}}{\partial t^2}, \bm{\varphi} \Big)_{\Omega_p} +  2 \nu_p \left( D(\bm{\eta}), D(\bm{\varphi}) \right)_{\Omega_p} + \lambda ( \nabla \cdot \bm{\eta}, \nabla \cdot \bm{\varphi} )_{\Omega_p} - \alpha(p_p, \nabla \cdot \bm{\varphi})_{\Omega_p} \\
&\hspace{15mm} - \langle g_1 \bm{n}_p, \bm{\varphi}\rangle_\gamma + ( g_2 \bm{\tau}_\gamma, \bm{\varphi})_\gamma = \langle\bm{f}_\eta, \bm{\varphi}\rangle_{\Omega_p} + \langle\bm{\eta}_N, \bm{\varphi}\rangle_{\widetilde{\Gamma}_N^p} \hspace{3mm} \forall \  \bm{\varphi} \in X, \\
&s_0 \Big( \frac{\partial p_p}{\partial t}, w\Big)_{\Omega_p} + \alpha\left( \nabla \cdot \frac{\partial \bm{\eta}}{\partial t},w \right)_{\Omega_p} + \kappa (\nabla p_p, \nabla w)_{\Omega_p} - (\lambda_p, w)_{\gamma}\\
& \hspace{15mm} = \langle f_p,w \rangle_{\Omega_p} + \langle p_N, w \rangle_{\Gamma^p_N} \hspace{3mm} \forall \ w \in Q_p.
\end{split}
\end{align}

\noindent The boundary integrals involving $g_1 \bm{n}_f$ and $ g_2 \bm{\tau}_\gamma$  in \eqref{FPSI_FEM:WF:subdomains} derive from \eqref{FPSI_FEM:Inter2}, which implies that $g_1\bm{n}_f + g_2 \bm{\tau}_\gamma = \bm{\sigma}_f \bm{n}_f = -\bm{\sigma}_p \bm{n}_p$. The three LMs allow us to rewrite the remaining interface conditions \eqref{FPSI_FEM:Inter1}, \eqref{FPSI_FEM:Inter3}, and \eqref{FPSI_FEM:Inter4} as  
\begin{align}\label{FPSI_FEM:interWith3LMS}
    \begin{split}
        \bm{u}\cdot \bm{n}_f + \frac{\partial \bm{\eta}}{\partial t} \cdot \bm{n}_p - \lambda_p = 0, \quad
        g_1 + p_p = 0, \quad
        g_2 + \beta \bm{u} \cdot \bm{\tau}_\gamma - \beta \frac{\partial \bm{\eta}}{\partial t} \cdot \bm{\tau}_\gamma= 0.
    \end{split}
\end{align}
\begin{remark}
The above conditions represent a restriction to the case $d=2$; however, the extension to $d=3$ would only require the definition of one more LM for the additional tangential direction (i.e., $g_2 = (\bm{\sigma}_f \bm{n}_f) \cdot \bm{\tau}_\gamma^1$ and $g_3 = (\bm{\sigma}_f \bm{n}_f) \cdot \bm{\tau}_\gamma^2$). In the analysis, this additional LM could be grouped with $g_2$ without affecting the structure of the proofs, and so we continue with the assumption $d=2$ for simplicity.
\end{remark}
\noindent We derive the weak form of \eqref{FPSI_FEM:interWith3LMS} by multiplying with test functions $s_1 \in \Lambda_{g1}, \mu \in \Lambda_\lambda$, and $s_2 \in \Lambda_{g2}$, respectively, and integrating:
\begin{align}\label{FPSI_FEM:interWith3LMS-WF}
    \begin{split}
        \langle \bm{u}\cdot \bm{n}_f,s_1\rangle_\gamma + \left\langle \frac{\partial \bm{\eta}}{\partial t} \cdot \bm{n}_p ,s_1 \right\rangle_\gamma - \langle \lambda_p,s_1\rangle_\gamma &= 0 \hspace{5mm} \forall \hspace{1mm} s_1 \in \Lambda_{g1},\\
        \langle g_1,\mu\rangle_\gamma + (p_p,\mu)_\gamma &= 0 \hspace{5mm} \forall \hspace{1mm} \mu \in \Lambda_\lambda ,\\
        (g_2,s_2)_\gamma + \beta (\bm{u} \cdot \bm{\tau}_\gamma,s_2)_\gamma - \beta \left(\frac{\partial \bm{\eta}}{\partial t} \cdot \bm{\tau}_\gamma,s_2\right)_\gamma &= 0 \hspace{5mm} \forall \hspace{1mm} s_2 \in \Lambda_{g2}.
    \end{split}
\end{align}

\noindent To derive the semi-discrete monolithic system, discretize \eqref{FPSI_FEM:WF:subdomains} and \eqref{FPSI_FEM:interWith3LMS-WF} in time using Backward Euler. To signify derivatives in time compactly, we adopt the notation 
\begin{equation}\label{FPSI_FEM:eq:TimeDerivNotation}
    \bm{\dot{\eta}}^{n} := \frac{\bm{\eta}^{n} - \bm{\eta}^{n-1}}{\Delta t}.
\end{equation}
Thus, the second derivative in time may be written as $ \bm{\ddot{\eta}}^{\hspace{0.1mm} n} :=  \dfrac{\bm{\dot{\eta}}^n - \bm{\dot{\eta}}^{n-1}}{\Delta t}$.

We move previous time step terms to the right hand side, scaling by factors of $\Delta t$ to clarify the saddle point (SP) structure. The time-discretized weak form for the FPSI system becomes:

\begin{align}\label{FPSI_FEM:WF:allTimeDisc}
    \begin{split}
       &\rho_f (  \bm{u}^{n+1} , \bm{v} )_{\Omega_f} + 2 \Delta t \nu_f \left( D(\bm{u}^{n+1}), D(\bm{v}) \right)_{\Omega_f} -  ( \Delta t p_f^{n+1}, \nabla \cdot \bm{v})_{\Omega_f} -  \langle \Delta t g_1^{n+1} \bm{n}_f,\bm{v}\rangle_\gamma \\
      &\hspace{15mm} -  ( \Delta t g_2^{n+1} \bm{\tau}_\gamma,\bm{v})_\gamma = \Delta t \langle\bm{f}_f^{n+1},\bm{v}\rangle_{\Omega_f} + \Delta t \langle\bm{u}_N^{n+1}, \bm{v}\rangle_{\Gamma_N^f} +\rho_f (  \bm{u}^n, \bm{v} )_{\Omega_f}  \hspace{3mm} \forall \ \bm{v} \in U \\
&( \nabla \cdot \bm{u}^{n+1}, q )_{\Omega_f} = 0 \hspace{3mm} \forall \ q \in Q_f \\
&\rho_p \left( \frac{\bm{\eta}^{n+1}}{\Delta t}, \bm{\varphi} \right)_{\Omega_p} +  2 \Delta t^2 \nu_p \left( D\left(\frac{\bm{\eta}^{n+1}}{\Delta t}\right), D(\bm{\varphi}) \right)_{\Omega_p} + \lambda \Delta t^2 \left( \nabla \cdot \left( \frac{\bm{\eta}^{n+1}}{\Delta t}\right), \nabla \cdot \bm{\varphi} \right)_{\Omega_p} \\
& \hspace{15mm}- \alpha  ( \Delta t p_p^{n+1}, \nabla \cdot \bm{\varphi})_{\Omega_p}- \langle \Delta t  g_1^{n+1}\bm{n}_p, \bm{\varphi}\rangle_\gamma + ( \Delta t  g_2^{n+1}\bm{\tau}_\gamma, \bm{\varphi})_\gamma  \\
&\hspace{15mm}= \Delta t \langle\bm{f}_\eta^{n+1}, \bm{\varphi}\rangle_{\Omega_p} + \Delta t \langle\bm{\eta}_N^{n+1}, \bm{\varphi}\rangle_{\widetilde{\Gamma}_N^p} + \rho_p \left( \frac{  2\bm{\eta}^n - \bm{\eta}^{n-1}}{\Delta t}, \bm{\varphi} \right)_{\Omega_p}\hspace{3mm} \forall \ \bm{\varphi} \in X \\
&\frac{s_0}{\Delta t^2} (\Delta t p_p^{n+1}, w)_{\Omega_p} + \alpha\left(\nabla \cdot \frac{\bm{\eta}^{n+1}}{\Delta t},w\right)_{\Omega_p} + \frac{\kappa}{\Delta t}  (\Delta t\nabla p_p^{n+1}, \nabla w)_{\Omega_p} -  ( \lambda_p^{n+1}, w)_{\gamma} \\
&\hspace{15mm}=  \langle f_p^{n+1},w \rangle_{\Omega_p} +  \langle p_N^{n+1}, w \rangle_{\Gamma_N^p} + \frac{s_0}{\Delta t^2} (\Delta t p_p^n, w)_{\Omega_p} + \alpha\left(\nabla \cdot \frac{\bm{\eta}^n}{\Delta t},w\right)_{\Omega_p} \hspace{3mm} \forall \ w \in Q_p \\
  & \langle \bm{u}^{n+1} \cdot \bm{n}_f,s_1\rangle_\gamma + \left\langle \frac{\bm{\eta}^{n+1} }{\Delta t} \cdot \bm{n}_p ,s_1 \right\rangle_\gamma   -  \langle \lambda_p^{n+1},s_1\rangle_\gamma = \left\langle \frac{ \bm{\eta}^n}{\Delta t} \cdot \bm{n}_p ,s_1 \right\rangle_\gamma \hspace{5mm} \forall \ s_1 \in \Lambda_{g1}\\
      &  \langle\Delta t g_1^{n+1},\mu\rangle_\gamma  + ( \Delta t p_p^{n+1},\mu)_\gamma = 0 \hspace{5mm} \forall \ \mu \in \Lambda_\lambda\\
      &  \frac{1}{\beta \Delta t} (\Delta t g_2^{n+1},s_2)_\gamma  +  (\bm{u}^{n+1} \cdot \bm{\tau}_\gamma,s_2)_\gamma - \left(\frac{\bm{\eta}^{n+1}}{\Delta t} \cdot \bm{\tau}_\gamma,s_2\right)_\gamma = - \left(\frac{\bm{\eta}^n}{\Delta t} \cdot \bm{\tau}_\gamma,s_2\right)_\gamma \hspace{5mm} \forall \ s_2 \in \Lambda_{g2}.
    \end{split}
\end{align}

\normalsize \noindent Note that the term $\dfrac{2\bm{\eta}^n - \bm{\eta}^{n-1}}{\Delta t}$ on the right-hand side of the third equation of \eqref{FPSI_FEM:WF:allTimeDisc} is equivalent to $(\bm{\eta}^n / \Delta t) + \bm{\dot{\eta}}^n$.  As initial conditions for both $\bm{\eta}$ and $\bm{\dot{\eta}}$ are provided, this term may be obtained from the initial conditions at the first time step, and on all subsequent time steps may be obtained by the formula given in \eqref{FPSI_FEM:eq:TimeDerivNotation}. The formulation \eqref{FPSI_FEM:WF:allTimeDisc} suggests the bilinear forms
\begin{align*} 
&a_1(\cdot,\cdot): \hspace{1mm} (U\times X) \times (U \times X) \mapsto \mathbb{R}, \quad a_2(\cdot,\cdot): \hspace{1mm} Q_p \times Q_p \mapsto \mathbb{R}, \quad a_g(\cdot,\cdot): \hspace{1mm} \Lambda_{g2} \times \Lambda_{g2} \mapsto \mathbb{R}, \\
&b_{\gamma,1}(\cdot,\cdot) : (U\times X) \times \Lambda_{g1} \mapsto \mathbb{R}, \quad b_{\gamma,2}(\cdot,\cdot) : \hspace{1mm} (U \times X) \times \Lambda_{g2} \mapsto \mathbb{R}, \\ 
&b_{LM}(\cdot,\cdot): \Lambda_{g1} \times \Lambda_\lambda \mapsto \mathbb{R}, \quad b_{pp}(\cdot,\cdot): X \times Q_p \mapsto \mathbb{R}, \quad b_{pf}(\cdot,\cdot):U\times Q_f \mapsto \mathbb{R},\\
& b_{1}(\cdot,\cdot): \hspace{1mm} (U \times X) \times (Q_f \times \Lambda_{g1}) \mapsto \mathbb{R}, \quad \text{ and } \quad b_2(\cdot,\cdot): Q_p \times \Lambda_\lambda \mapsto \mathbb{R},
\end{align*}
where 
\begin{align*}
\begin{split}
a_1\left( (\bm{u},\bm{\eta}), (\bm{v}, \bm{\varphi}) \right)  &:= \rho_f (\bm{u},\bm{v})_{\Omega_f} + 2 \nu_f \Delta t (D(\bm{u}),D(\bm{v}))_{\Omega_f} + \rho_p(\bm{\eta},\bm{\varphi})_{\Omega_p}  \\& \hspace{5mm} + 2 \nu_p \Delta t^2  (D(\bm{\eta}),D(\bm{\varphi}))_{\Omega_p} + \Delta t^2 \lambda (\nabla \cdot \bm{\eta}, \nabla \cdot \bm{\varphi})_{\Omega_p}, \\
a_2(p_p, w) &:= \frac{s_0}{\Delta t^2} (p_p,w)_{\Omega_p} + \frac{\kappa}{\Delta t}  (\nabla p_p, \nabla w)_{\Omega_p}\\
a_g(g_2,s_2) &:= \frac{1}{\beta \Delta t} (g_2,s_2)_\gamma,
\end{split} 
\end{align*}

and the mixed terms are defined as
\begin{align*}
\begin{split}
b_{\gamma,1}\left( (\bm{v},\bm{\varphi}) ,s_1\right) &:= -\langle\bm{\varphi},s_1\bm{n}_p\rangle_\gamma -\langle\bm{v},s_1 \bm{n}_f\rangle_\gamma, \quad b_{\gamma,2}\left( (\bm{v},\bm{\varphi}),s_2\right) := (\bm{\varphi},s_2\bm{\tau}_\gamma)_\gamma - (\bm{v},s_2 \bm{\tau}_\gamma)_\gamma \\
b_{LM}(s_1,\mu) &:= \langle   s_1, \mu \rangle_\gamma, \quad b_{pp}(\bm{\varphi},w) := -\alpha (\nabla \cdot \bm{\varphi},w)_{\Omega_p}, \quad 
b_{pf}(\bm{v},q) := -(\nabla \cdot \bm{v}, q)_{\Omega_f} \\
b_{1}\left((\bm{v}, \bm{\varphi}), (q, s_1) \right) &:= b_{\gamma,1} \left( (\bm{v},\bm{\varphi}), s_1\right) + b_{pf}(\bm{v},q), \quad  b_2(w,\mu) := (w,\mu)_\gamma. \\
\end{split}
\end{align*}
To correspond to the weak form, define the scaled variables $\widehat{\bm{\eta}}^{\hspace{0.5mm} n+1}:= \bm{\eta}^{n+1} / \Delta t \in X$, $\widehat{p}_f^{\hspace{0.7mm} n+1} := \Delta t p_f^{n+1} \in Q_f$, $\widehat{p}_p^{\hspace{0.7mm} n+1} := \Delta t p_p^{n+1} \in Q_p$, $\widehat{g}_i^{\hspace{0.7mm} n+1} := \Delta t g_i^{n+1} \in \Lambda_{gi}, \hspace{1mm} \text{ for } i \in \{1,2\}$.  Using these bilinear forms and scaled variables, the Stokes-Biot system \eqref{FPSI_FEM:WF:allTimeDisc} can be represented as follows.

\noindent \textit{Find $\bm{u}^{n+1} \in U, \widehat{p}_f^{\hspace{0.7mm} n+1}  \in Q_f,$ $\widehat{\bm{\eta}}^{\hspace{0.5mm} n+1} \in X, \widehat{p}_p^{\hspace{0.7mm} n+1} \in Q_p, \hspace{0.7mm} \widehat{g}_1^{\hspace{0.7mm} n+1} \in \Lambda_{g1} $, $ \widehat{g}_2^{\hspace{0.7mm} n+1} \in \Lambda_{g2}$, and $\lambda_p \in \Lambda_\lambda$   s.t. $\forall (\bm{v},\bm{\varphi}) \in U \times X, w \in Q_p, (q,s_1) \in Q_f \times \Lambda_{g1}, \mu \in \Lambda_\lambda, s_2 \in \Lambda_{g2}$}
\refstepcounter{equation}\label{FPSI_FEM:eq:bigSaddlePointSystem}
\begin{align}
     &a_1\left( (\bm{u}^{n+1}, \widehat{\bm{\eta}}^{\hspace{0.5mm} n+1}), (\bm{v}, \bm{\varphi})\right) + b_1((\bm{v},\bm{\varphi}),(\widehat{p}_f^{\hspace{0.5mm} n+1} ,\widehat{g}_1^{\hspace{0.7mm} n+1})) +  b_{\gamma,2}((\bm{v},\bm{\varphi}),\widehat{g}_2^{\hspace{0.7mm} n+1}) \notag \\
       &\hspace{15mm}  + b_{pp}(\bm{\varphi},\widehat{p}_p^{\hspace{0.7mm} n+1}) = \mathcal{F}_1(\bm{v},\bm{\varphi}) \tag{\theequation (a)}\label{FPSI_FEM:eq:bigSaddlePointSystem_a}\\ 
      & a_2 \left(\widehat{p}_p^{\hspace{0.7mm} n+1},w\right) - b_2(w, \lambda_p^{n+1})  -b_{pp}(\widehat{\bm{\eta}}^{\hspace{0.5mm} n+1},w) = \mathcal{F}_2(w)  \tag{\theequation (b)}\label{FPSI_FEM:eq:bigSaddlePointSystem_b}\\ 
     &  - b_1((\bm{u}^{n+1},\widehat{\bm{\eta}}^{\hspace{0.5mm} n+1}),(q,s_1)) -  b_{LM}(s_1,\lambda_p^{n+1}) = \mathcal{F}_3(q,s_1)  \tag{\theequation (c)}\label{FPSI_FEM:eq:bigSaddlePointSystem_c} \\ 
     &  b_{LM}(\widehat{g}_1^{\hspace{0.7mm} n+1},\mu)   + b_2( \widehat{p}_p^{\hspace{0.7mm} n+1},\mu) = 0 \tag{\theequation (d)}\label{FPSI_FEM:eq:bigSaddlePointSystem_d}\\ 
      & a_g(\widehat{g}_2^{\hspace{0.7mm} n+1},s_2) - b_{\gamma,2}\left( (\bm{u}^{n+1},\widehat{\bm{\eta}}^{\hspace{0.5mm} n+1}),s_2\right) = \mathcal{F}_4(s_2)  \tag{\theequation (e)}\label{FPSI_FEM:eq:bigSaddlePointSystem_e} 
\end{align}
with right hand sides defined by
 \begin{align*}
\mathcal{F}_1(\bm{v},\bm{\varphi}) &= \Delta t\langle \bm{f}_f^{n+1},\bm{v}\rangle_{\Omega_f}  + \Delta t\langle\bm{u}_N^{n+1}, \bm{v}\rangle_{\Gamma_N^f}  + \rho_f (\bm{u}^n, \bm{v})_{\Omega_f} + \Delta t \langle\bm{f}_\eta^{n+1}, \bm{\varphi}\rangle_{\Omega_p}\\
&\hspace{20mm} + \Delta t \langle\bm{\eta}_N^{n+1}, \bm{\varphi}\rangle_{\widetilde{\Gamma}_N^p}  + \rho_p (2\widehat{\bm{\eta}}^{\hspace{0.5mm} n} - \widehat{\bm{\eta}}^{\hspace{0.5mm} n-1},\bm{\varphi})_{\Omega_p}\\
\mathcal{F}_2(w) &=  \langle f_p^{n+1},w \rangle_{\Omega_p} +  \langle p_N^{n+1}, w \rangle_{\Gamma_N^p} + \frac{s_0}{\Delta t^2} (\widehat{p_p}^n, w)_{\Omega_p} + \alpha\left(\nabla \cdot \widehat{\bm{\eta}}^n,w\right)_{\Omega_p} \\
\mathcal{F}_3(q,s_1) &= \langle \widehat{\bm{\eta}}^{\hspace{0.5mm} n}\cdot \bm{n}_p,s_1 \rangle_\gamma \\
\mathcal{F}_4(s_2) &= -(\widehat{\bm{\eta}}^{\hspace{0.5mm} n} \cdot \bm{\tau}_\gamma,s_2)_\gamma.
\end{align*}

\subsection{Saddle Point Structure and Well-Posedness}\label{FPSI_FEM:subsec:BilinearForms}
The system \eqref{FPSI_FEM:eq:bigSaddlePointSystem} clearly displays a saddle point structure. 
Since the analysis of this system depends on how the variables are grouped, we present several grouping options and provide justification for our selected approach.


First, as both the displacement and pore pressure are in $H^1(\Omega_p)$, we choose not to treat the term $b_{pp}(\cdot,\cdot)$ as a mixed term. Establishing an inf-sup condition for $b_{\gamma, 2}(\cdot,\cdot)$ between the spaces $(U \times X)$ and $\Lambda_{g2}$ would require the trace operators $ \bm{u} \mapsto \bm{u}|_\gamma, \bm{\eta} \mapsto \bm{\eta}|_\gamma$ to be surjective from $U$ and $X$ onto the dual of $\Lambda_{g2}$.  This, in turn, would imply that $\Lambda_{g2} = H^{-1/2}(\gamma)$. However, due to regularity requirements stemming from the term $a_g(\cdot,\cdot)$, the LM $g_2$ must have at least $L^2(\gamma)$ regularity. Consequently, $\Lambda_{g2} $ is defined to be $L^2(\gamma)$ and cannot be chosen as $H^{-1/2}(\gamma)$.
 Therefore, it remains unclear whether an inf-sup condition can be established between $(U \times X)$ and $\Lambda_{g2}$.  Based on these considerations, we group the variables $p_p$ and $g_2$ with $\bm{u}$ and $\bm{\eta}$, so that the terms $b_{pp}(\cdot,\cdot)$ and $b_{\gamma,2}(\cdot,\cdot)$ are treated with the coercive parts of the system rather than as mixed terms. To this end, define the spaces $Y := U \times X \times Q_p \times \Lambda_{g2}$, and $Z := Q_f \times \Lambda_{g1}$, with norms
\begin{align*}
||\bm{y}||_{Y}^2 &= ||(\bm{u},\bm{\eta},p_p,g_2)||_{Y}^2 := ||\bm{u}||_{1,\Omega_f}^2 + ||\bm{\eta}||_{1,\Omega_p}^2 + ||p_p||^2_{1,\Omega_p} + ||g_2||^2_{0,\gamma} \\
||\bm{z}||_{Z}^2 &= ||(p_f,g_1)||_{Z}^2 := ||p_f||_{0,\Omega_f}^2 + ||g_1||_{-1/2,\gamma}^2.
\end{align*}

\noindent We combine several of the bilinear forms between functions in $Y$ to create $a_Y(\cdot,\cdot) : Y \times Y \mapsto \mathbb{R}$:
\begin{align*}
   a_Y\Big((\bm{u},\bm{\eta},p_p,g_2),(\bm{v},\bm{\varphi},w,s_2)\Big) := a_1\Big( (\bm{u},\bm{\eta}),(\bm{v},\bm{\varphi})\Big)  +a_2(p_p,w)  + a_g(g_2,s_2)  \\
    +b_{\gamma,2}((\bm{v},\bm{\varphi}),g_2) + b_{pp}(\bm{\varphi},p_p) - b_{pp}(\bm{\eta},w) - b_{\gamma,2}((\bm{u},\bm{\eta}),s_2), 
\end{align*}

\noindent which simplifies the structure of \eqref{FPSI_FEM:eq:bigSaddlePointSystem}   for all $(\bm{v},\bm{\varphi},w,s_2) \in Y, \mu \in \Lambda_\lambda$, and $(q,s_1) \in Z$ to:
\refstepcounter{equation}\label{FPSI_FEM:eq:bppWithCoercive_BLMmixed}
\begin{align}
        &a_Y\big( (\bm{u}^{n+1}, \widehat{\bm{\eta}}^{\hspace{0.5mm} n+1},\widehat{p}_p^{\hspace{0.7mm} n+1},\widehat{g}_2^{\hspace{0.7mm} n+1}), (\bm{v}, \bm{\varphi},w,s_2)\big) +b_1((\bm{v},\bm{\varphi}),(\widehat{p}_f^{\hspace{0.7mm} n+1},\widehat{g}_1^{\hspace{0.7mm} n+1})) \notag \\  
        & \hspace{10mm} -b_2(w,\lambda_p^{n+1})= \mathcal{F}_1(\bm{v},\bm{\varphi}) +\mathcal{F}_2(w) +\mathcal{F}_4(s_2) \tag{\theequation (a)}\label{FPSI_FEM:eq:bppWithCoercive_BLMmixeda} \\
           & b_{LM}(\widehat{g}_1^{\hspace{0.7mm} n+1},\mu)   + b_2( \widehat{p}_p^{\hspace{0.7mm} n+1},\mu) = 0 \tag{\theequation (b)}\label{FPSI_FEM:eq:bppWithCoercive_BLMmixedb}\\ 
&  - b_1((\bm{u}^{n+1},\widehat{\bm{\eta}}^{\hspace{0.5mm} n+1}),(q,s_1))  -  b_{LM}(s_1,\lambda_p^{n+1}) = \mathcal{F}_3(q,s_1) \tag{\theequation (c)}\label{FPSI_FEM:eq:bppWithCoercive_BLMmixedc} 
\end{align}

At this stage, there are several options for how to define the space $\Lambda_\lambda$, with their structures illustrated visually in Table \ref{FPSI_FEM:tab:Possible SP structures for analysis}.

In the first mixed formulation, MF(1), $\Lambda_\lambda$ is grouped with $Z$ so that \eqref{FPSI_FEM:eq:bppWithCoercive_BLMmixed} becomes a single saddle point system. The mixed term $-b_2(\cdot,\cdot)+b_1(\cdot,\cdot)$ would then be required to satisfy an inf-sup condition between $Y$ and $\Lambda_\lambda \times Z$. As part of this condition, one would need to show the surjectivity of the trace operators $ \bm{u} \mapsto \bm{u}|_\gamma, \bm{\eta} \mapsto \bm{\eta}|_\gamma$, and $p_p \mapsto p_p |_\gamma$ from $U,X,$ and $Q_p$ onto the duals of $\Lambda_{g1}$ and $\Lambda_\lambda$ (\cite{Chen_2021, Gatica_2011_structure}). However, this requires both $\Lambda_\lambda$ and $\Lambda_{g1}$ to be $H^{-1/2}(\gamma)$, which would make the term $b_{LM}(\cdot,\cdot)$ undefined as it is an $L^2$ inner product. Additionally, the lower $2 \times 2$ block of the system would need to be positive semi-definite (\cite{Gatica_2011_structure}), and this grouping of variables lacks coercive terms on the diagonals of that block. 

\begin{small}
\begin{table}[h]
\centering
\subfloat[Single SP system; $\Lambda_\lambda$ with $Z$]{
\begin{tabular}{|c|c||cc|}
    \hline 
MF(1) & $\bm{y}$ & $\lambda_p$ & $p_f,g_1$ \\
\hline 
\ref{FPSI_FEM:eq:bppWithCoercive_BLMmixeda}   & $a_Y$ & $-b_2$ & $b_1$   \\ \hline 
\ref{FPSI_FEM:eq:bppWithCoercive_BLMmixedb} & $b_2$ &  & $b_{LM}$ \\ 
\ref{FPSI_FEM:eq:bppWithCoercive_BLMmixedc} & $-b_1$ & $-b_{LM}$ & \\ \hline 
\end{tabular}}
\qquad 
\subfloat[Double SP system]{\begin{tabular}{|c|c|c||c|}
    \hline 
MF(2) & $\bm{y}$ & $\lambda_p$ & $p_f,g_1$ \\
\hline 
\ref{FPSI_FEM:eq:bppWithCoercive_BLMmixeda}   & $a_Y$ & $-b_2$ & $b_1$   \\ \hline
\ref{FPSI_FEM:eq:bppWithCoercive_BLMmixedb} & $b_2$ &  & $b_{LM}$ \\ \hline 
\ref{FPSI_FEM:eq:bppWithCoercive_BLMmixedc} & $-b_1$ & $-b_{LM}$ & \\ \hline 
\end{tabular}}
\\
\subfloat[Single SP system; $\Lambda_\lambda$ with $Y$]{
\begin{tabular}{|c|cc||c|}
    \hline 
MF(3) & $\bm{y}$ & $\lambda_p$ & $p_f,g_1$ \\
\hline 
\ref{FPSI_FEM:eq:bppWithCoercive_BLMmixeda}   & $a_Y$ & $-b_2$ & $b_1$   \\ 
\ref{FPSI_FEM:eq:bppWithCoercive_BLMmixedb} & $b_2$ &  & $b_{LM}$ \\ \hline 
\ref{FPSI_FEM:eq:bppWithCoercive_BLMmixedc} & $-b_1$ & $-b_{LM}$ & \\ \hline 
\end{tabular}}
\qquad 
\subfloat[Stabilized SP system; $\Lambda_\lambda$ with $Y$]{
\begin{tabular}{|c|cc||c|}
    \hline 
MF(4) & $\bm{y}$ & $\lambda_p$ & $p_f,g_1$ \\
\hline 
\ref{FPSI_FEM:eq:bppWithCoercive_BLMmixeda}   & $a_Y$ & $-b_2$ & $b_1$   \\ 
\ref{FPSI_FEM:eq:bppWithCoercive_BLMmixedb} & $b_2$ &$a_\lambda$ & $b_{LM}$ \\ \hline 
\ref{FPSI_FEM:eq:bppWithCoercive_BLMmixedc} & $-b_1$ & $-b_{LM}$ & \\ \hline 
\end{tabular}}
\caption{Possible groupings for FPSI saddle point system}
\label{FPSI_FEM:tab:Possible SP structures for analysis}
\end{table}
\end{small}

With our choices for the spaces $\Lambda_\lambda$ and $\Lambda_{g1}$, it is straightforward to show inf-sup conditions between $Y \times \Lambda_\lambda$ and $Z$ for the terms $b_1(\cdot,\cdot)$ and $b_{LM}(\cdot,\cdot)$. This leads to two options for the treatment of $\Lambda_\lambda$. One option is to formulate system \eqref{FPSI_FEM:eq:bppWithCoercive_BLMmixed} as a double saddle point system, as shown in MF(2), where an inf-sup condition for 
$b_2(\cdot,\cdot)$ between $Y$ and $\Lambda_\lambda$ must hold. However, this again appears to require setting $\Lambda_\lambda = H^{-1/2}(\gamma)$, which would, in turn, demand higher regularity for 
$\Lambda_{g1}$ and introduce complications in proving the outer inf-sup conditions for $b_1(\cdot,\cdot)$ and $b_{LM}(\cdot,\cdot)$.  

We therefore consider MF(3), where the system is once again formulated as a single saddle point problem, but now with $\Lambda_\lambda$ with $Y$ instead of $Z$. 
The only problem with this formulation is the lack of a coercive term for $\lambda_p$. To address this, we propose a modification to the formulation: the inclusion of a stabilization term
$  a_\lambda(\cdot,\cdot):$ $\Lambda_\lambda \times \Lambda_\lambda \mapsto \mathbb{R}$, defined by $ a_\lambda(\lambda_p^{n+1},\mu) = \overline{\epsilon} (\lambda_p^{n+1},\mu)_{1/2,\gamma}$ in equation \eqref{FPSI_FEM:eq:bppWithCoercive_BLMmixedb}, where $\overline{\epsilon}$ is a small, positive number. This modified formulation is shown in MF(4).  We note that $H^{1/2}(\gamma)$ is a Hilbert space equivalent to the Sobolev space $W^{1/2,2}(\gamma)$, whose inner product may be defined with a H\"{o}lder-like seminorm for the fractional derivative term (\cite{Chen_Sobolev, Lions_2012}).

Although included for different purposes, the term $a_\lambda(\cdot,\cdot)$ retains some similarities to the term added in the fluid pressure Laplacian (FPL) technique (\cite{Cesmelioglu_2020, Lee_2023}). In the FPL technique, a penalty term is added to the fully discrete weak formulation in order to eradicate spurious pressure oscillations which may occur in the Biot problem even with inf-sup stable spaces for low permeability or low compressibility. With a large pressure gradient, poroelastic locking can still be a concern, and thus numerical diffusion is added to the Darcy pressure by including a penalty term of the form $\epsilon h^2 (\nabla \dot{p}_{p,h},\nabla w)$. Recalling the definition of $\lambda_p$ as $\lambda_p := \kappa \nabla p_p \cdot \bm{n}_p$ on $\gamma$, the term $a_\lambda(\cdot,\cdot)$ can be viewed as an inner product between functions representing pressure gradients restricted from $H^1(\Omega_p)$ to $H^{1/2}(\gamma)$. While $a_\lambda(\cdot,\cdot)$ is used in our work for the semi-discrete weak form and is an inner product over an interface instead of a subdomain, we mention the FPL technique here as it is reminiscent of our stabilization term.

\begin{remark}
The stabilization term could be avoided by originally grouping $\Lambda_{g2}$ with $\Lambda_\lambda$, giving rise to a double saddle point structure where the smaller saddle point system is between $(U \times X \times Q_p)
$ and $(\Lambda_{g2} \times \Lambda_\lambda)$. In this scenario, the presence of the positive semi-definite term $a_{g2}(\cdot,\cdot)$ would negate the need for the stabilization. However, it would then be necessary to prove inf-sup conditions for $b_{\gamma 2}(\cdot,\cdot)$ and $b_2(\cdot,\cdot)$ with respect to that smaller saddle point system. For the fully discrete formulation this technique would yield well-posedness; however, as explained above, treating $b_{\gamma 2}$ as a mixed term is not a viable option for the continuous model. 
\end{remark}

Numerical results suggest that the stabilization term $a_\lambda(\cdot,\cdot)$ is not needed in practice. However, showing well-posedness of the continuous formulation with $\overline{\epsilon} = 0$ remains an open problem. For this work, we continue with the analysis of MF(4), assuming that $\overline{\epsilon} > 0$. Grouping spaces together one final time, we define $M:= Y \times \Lambda_\lambda$, with norm $||\bm{m}||_M^2 = ||(\bm{y},\mu)||^2_M := ||\bm{y}||_Y^2 + ||\mu||_{1/2,\gamma}^2$. Using the structure of MF(4), \eqref{FPSI_FEM:eq:bppWithCoercive_BLMmixed} can be expressed as:

\noindent \textit{Find $\bm{m}^{n+1} := (\bm{u}^{n+1},\widehat{\bm{\eta}}^{\hspace{0.5mm} n+1},\widehat{p}_p^{\hspace{0.7mm} n+1},\widehat{g}_2^{\hspace{0.7mm} n+1}, \lambda_p^{n+1}) \in M$, $ \bm{z}^{n+1} := (\widehat{p}_f^{\hspace{0.7mm} n+1},\widehat{g}_1^{\hspace{0.7mm} n+1}) \in Z$ such that:}
\begin{align}\label{FPSI_FEM:MF4toUse}
    \begin{split}
       a_M\left( \bm{m}^{n+1}, \bm{\zeta} \right) +b_{MZ}(\bm{\zeta},\bm{z}^{n+1}) 
        &= \mathcal{F}_M(\bm{\zeta}) \hspace{5mm} \forall \hspace{1mm} \bm{\zeta} := (\bm{v},\bm{\varphi},w,s_2,\mu) \in M,\\
b_{MZ}(\bm{m}^{n+1},\bm{\xi}) &= -\mathcal{F}_3(\bm{\xi}) \hspace{5mm} \forall \hspace{1mm} \bm{\xi} := (q,s_1) \in Z ,\\
    \end{split}
\end{align}
where 
\begin{align*}
   a_M\left(\bm{m}^{n+1}, \bm{\zeta} \right)  &:=  a_Y\left( (\bm{u}^{n+1}, \widehat{\bm{\eta}}^{\hspace{0.5mm} n+1},\widehat{p}_p^{\hspace{0.7mm} n+1},\widehat{g}_2^{\hspace{0.7mm} n+1}), (\bm{v}, \bm{\varphi},w,s_2)\right) \\
   &- b_2(w,\lambda_p^{n+1}) + b_2(\widehat{p}_p^{\hspace{0.7mm} n+1},\mu) + a_\lambda(\lambda_p^{n+1},\mu) \\ 
    b_{MZ}(\bm{\zeta},\bm{\xi}) &:= b_1((\bm{v},\bm{\varphi}),(q,s_1)) +    b_{LM}(s_1,\mu)\\
    \mathcal{F}_M(\bm{\zeta}) &:= \mathcal{F}_1(\bm{v},\bm{\varphi}) +\mathcal{F}_2(w)+\mathcal{F}_4(s_2).
 \end{align*}

To show the well-posedness of the saddle point system \eqref{FPSI_FEM:MF4toUse},  $a_M(\cdot,\cdot)$ should be coercive on the kernel of $b_{MZ}(\cdot,\cdot)$ in $M$, and the inf-sup condition for $b_{MZ}(\cdot,\cdot)$ must hold between $M$ and $Z$ (\cite{Brezzi_1990}). It is straightforward to show the continuity of $\mathcal{F}_M$ and $\mathcal{F}_3$, which are defined in terms of given forcing functions, Neumann conditions, and previous time step solutions. We begin by showing the inf-sup condition.

\begin{theorem}\label{FPSI_FEM:lem:Binfsup}
    There exists a positive constant $\beta_1 > 0$ such that $$\underset{\bm{0} \neq \bm{\zeta} \in M}{\sup} \frac{b_{MZ}(\bm{\zeta},\bm{\xi})}{||\bm{\zeta}||_{M}} \geq \beta_1 ||\bm{\xi}||_{Z} \hspace{3mm} \forall \bm{\xi} \in Z.$$
\end{theorem}

\begin{proof}
    Let $ \widetilde{\bm{\xi}} = (\widetilde{q},\widetilde{s}_1) \in Z$ be given. As $ \Lambda_{g1} = H^{-1/2}(\gamma)$ is a Hilbert space, by the Riesz Representation Theorem, we may find $\mu^* \in H^{1/2}(\gamma) = \Lambda_\lambda$ such that 
    \begin{equation}\label{FPSI_FEM:eq:RieszRepContinuous}
        || \mu^*||_{1/2,\gamma} = ||\widetilde{s}_1||_{-1/2,\gamma}, \text{ and } \langle \widetilde{s}_1,\theta \rangle_\gamma = (\mu^*,\theta)_{1/2,\gamma} \hspace{3mm} \forall \hspace{1mm} \theta \in H^{1/2}(\gamma).
        \end{equation}
    
 \noindent With $\mu^* \bm{n}_f \in \bm{H^{1/2}}(\gamma)$, by Lemma 2.1 in (\cite{deCastro_CAMWA_2025}), we can find a $\widetilde{\bm{u}} \in U$ s.t.
\begin{align}\label{FPSI_FEM:eq:UTrace_InfSup}
\begin{split}
b_{pf}(\widetilde{\bm{u}},q) &= -(\nabla \cdot \widetilde{\bm{u}}, q)_{\Omega_f} = (\widetilde{q},q)_{\Omega_f} \hspace{5mm} \forall q \in Q_f,  \\
\widetilde{\bm{u}}\big|_\gamma &= -\mu^* \bm{n}_f, \\
||\widetilde{\bm{u}}||_1 &\leq C_1(||\widetilde{q}||_0 + ||\mu^* \bm{n}_f||_{1/2,\gamma}). 
\end{split}
\end{align}
\noindent Our assumption that the measure of $\Gamma^f_N$ is nonzero is needed for the application of the divergence theorem in this lemma. By the same lemma, we can find $\widetilde{\bm{\eta}} \in X$ s.t. 
\begin{equation}\label{FPSI_FEM:eq:EtaTrace_InfSup}
    -(\nabla \cdot \widetilde{\bm{\eta}},w)_{\Omega_p} = 0 \ \forall \ w \in Q_p, \qquad \widetilde{\bm{\eta}}\big|_\gamma = -\mu^* \bm{n}_p, \qquad ||\widetilde{\bm{\eta}}||_1 \leq C_2 || \mu^* \bm{n}_p||_{1/2,\gamma}.
\end{equation}
Combining inequalities, 
\begin{align}\label{FPSI_FEM:eq:UEMU_bound}
||\widetilde{\bm{u}}||_{1,\Omega_f} + ||\widetilde{\bm{\eta}}||_{1,\Omega_p} + ||\mu^*||_{1/2,\gamma} &\leq \overline{C} (||\mu^* \bm{n}_p|| _{1/2,\gamma} + ||\mu^* \bm{n}_f||_{1/2,\gamma} + ||\widetilde{q}||_{0,\Omega_f} + ||\mu^* ||_{1/2,\gamma} ). 
\end{align}
Let $\widetilde{\bm{m}} := (\widetilde{\bm{u}},\widetilde{\bm{\eta}},0,0,\mu^*)$ and apply the trace properties in \eqref{FPSI_FEM:eq:UTrace_InfSup} and \eqref{FPSI_FEM:eq:EtaTrace_InfSup}:
    \begin{align*}
        \underset{0\neq\bm{\zeta} \in M}{\sup} \frac{b_{MZ}(\bm{\zeta},\widetilde{\bm{\xi}})}{||\bm{\zeta}||_M} &\geq     \frac{b_{MZ}(\widetilde{\bm{m}},\widetilde{\bm{\xi}})}{||\widetilde{\bm{m}}||_M}   = \frac{b_{1}\left( (\widetilde{\bm{u}},\widetilde{\bm{\eta}}), ( \widetilde{q},\widetilde{s}_1)  \right)+b_{LM}(\widetilde{s}_1,\mu^*)}{||\widetilde{\bm{m}}||_M} \\
        &= \frac{-\langle\widetilde{\bm{\eta}},\widetilde{s}_1\bm{n}_p\rangle_\gamma - \langle \widetilde{\bm{u}},\widetilde{s}_1\bm{n}_f\rangle_\gamma - (\nabla \cdot \widetilde{\bm{u}},\widetilde{q})_{\Omega_f} + \langle \widetilde{s}_1,\mu^*\rangle_\gamma }{||\widetilde{\bm{m}}||_M}\\
    &= \frac{\langle \mu^* \bm{n}_p,\widetilde{s}_1\bm{n}_p\rangle_\gamma + \langle \mu^* \bm{n}_f,\widetilde{s}_1\bm{n}_f\rangle_\gamma - (\nabla \cdot \widetilde{\bm{u}},\widetilde{q})_{\Omega_f} +\langle \widetilde{s}_1,\mu^*\rangle_\gamma }{||\widetilde{\bm{m}}||_M}.
    \end{align*}
Using the identities in \eqref{FPSI_FEM:eq:RieszRepContinuous} and \eqref{FPSI_FEM:eq:UTrace_InfSup}, we continue the equality as:
    \begin{align*}
 &= \frac{((\mu^* \bm{n}_p)\cdot \bm{n}_p,\mu^*)_{1/2,\gamma} + ((\mu^* \bm{n}_f)\cdot\bm{n}_f,\mu^*)_{1/2,\gamma} - (\nabla \cdot \widetilde{\bm{u}},\widetilde{q})_{\Omega_f} +( \mu^*,\mu^*)_{1/2,\gamma} }{||\widetilde{\bm{m}}||_M}\\
 &= \frac{||\mu^* \bm{n}_p||^2_{1/2,\gamma} + ||\mu^* \bm{n}_f||^2_{1/2,\gamma} + ||\widetilde{q}||_{0,\Omega_f}^2 +||\mu^*||^2_{1/2,\gamma} }{||\widetilde{\bm{m}}||_M}.
 \end{align*}
 We finish by employing algebraic inequalities 
 and the bound in \eqref{FPSI_FEM:eq:UEMU_bound}
 \begin{align*}
 \underset{0\neq\bm{\zeta} \in M}{\sup} \frac{b_{MZ}(\bm{\zeta},\widetilde{\bm{\xi}})}{||\bm{\zeta}||_M}  &\geq \frac{1}{||\widetilde{\bm{m}}||_M}  \frac{1}{4}\Big(||\mu^* \bm{n}_p||_{1/2,\gamma} + ||\mu^* \bm{n}_f||_{1/2,\gamma} + ||\widetilde{q}||_{0,\Omega_f} +||\mu^*||_{1/2,\gamma}\Big)^2  \\
   &\geq   \frac{1}{4 \overline{C}||\widetilde{\bm{m}}||_M}  \Big(||\widetilde{\bm{u}}||_{1,\Omega_f} + ||\widetilde{\bm{\eta}}||_{1,\Omega_p} + ||\mu^*||_{1/2,\gamma} \Big) \\
   & \hspace{25mm} \Big(||\mu^* \bm{n}_p||_{1/2,\gamma} + ||\mu^* \bm{n}_f||_{1/2,\gamma} + ||\widetilde{q}||_{0,\Omega_f} +||\mu^*||_{1/2,\gamma}\Big)\\
    &\geq \frac{1}{4 \overline{C}||\widetilde{\bm{m}}||_M} \Big(||\widetilde{\bm{u}}||_{1,\Omega_f} + ||\widetilde{\bm{\eta}}||_{1,\Omega_p} + ||\mu^*||_{1/2,\gamma} \Big)\Big( ||\widetilde{q}||_{0,\Omega_f} +||\widetilde{s}_1||_{-1/2,\gamma}\Big)\\
  &\geq \frac{1}{4\overline{C}} \frac{\Big(||\widetilde{\bm{u}}||_{1,\Omega_f}^2 + ||\widetilde{\bm{\eta}}||_{1,\Omega_p}^2 + ||\mu^*||_{1/2,\gamma}^2 \Big)^{1/2} \Big( ||\widetilde{q}||^2_{0,\Omega_f} +||\widetilde{s}_1||^2_{-1/2,\gamma}\Big)^{1/2} }{\Big(|| \widetilde{\bm{u}} ||_{1,\Omega_f}^2 + ||\widetilde{\bm{\eta}}||_{1,\Omega_p}^2 + ||\mu^*||_{1/2,\gamma}^2 \Big)^{1/2}} \\
    &= \frac{1}{4\overline{C}} \Big( ||\widetilde{q}||_{0,\Omega_f}^2 +||\widetilde{s}_1||_{-1/2,\gamma}^2 \Big)^{1/2} .
    \end{align*}
    Thus, as $\bm{\xi} \in Z$ is arbitrary, $$ \underset{0\neq\bm{\zeta} \in M}{\sup} \frac{b_{MZ}(\bm{\zeta},\widetilde{\bm{\xi}})}{||\bm{\zeta}||_M} \geq \frac{1}{4\overline{C}} || \widetilde{\bm{\xi}}||_Z$$ implies the desired condition.
\end{proof}

\noindent Next, we show that $a_M(\cdot,\cdot)$ is coercive on $M$. 

\begin{lemma}\label{FPSI_FEM:lem:AM_ell}
    The bilinear form $a_M(\cdot,\cdot)$ is coercive, i.e. there exists an $\alpha_1 > 0 $ such that \\
    $ a_M(\bm{m},\bm{m}) \geq \alpha_1 ||\bm{m}||^2_M \quad \forall \ \bm{m} \in M$.
\end{lemma}
\begin{proof}
     By definition, for $\bm{m} = (\bm{v},\bm{\varphi},w,s_2,\mu) \in M$, 
    \begin{align*}
        a_M&\Big( (\bm{v},\bm{\varphi},w,s_2,\mu),(\bm{v},\bm{\varphi},w,s_2,\mu) \Big) = \rho_f (\bm{v},\bm{v})_{\Omega_f} + 2 \nu_f \Delta t (D(\bm{v}),D(\bm{v}))_{\Omega_f}  + \rho_p(\bm{\varphi},\bm{\varphi})_{\Omega_p}\\
        &+ 2 \nu_p \Delta t^2  (D(\bm{\varphi}),D(\bm{\varphi}))_{\Omega_p}   + \Delta t^2 \lambda (\nabla \cdot \bm{\varphi}, \nabla \cdot \bm{\varphi})_{\Omega_p} +\frac{s_0}{\Delta t^2} (w,w)_{\Omega_p} + \frac{\kappa}{\Delta t} (\nabla w,\nabla w)_{\Omega_p}  \\
        &+  \frac{1}{\beta \Delta t}(s_2,s_2)_\gamma  + (\bm{\varphi},s_2\bm{\tau}_\gamma)_\gamma - (\bm{v},s_2\bm{\tau}_\gamma)_\gamma
         - \alpha(\nabla \cdot \bm{\varphi},w)_{\Omega_p}    +  \alpha(\nabla \cdot \bm{\varphi},w)_{\Omega_p}\\
         & - (\bm{\varphi},s_2\bm{\tau}_\gamma)_\gamma + (\bm{v},s_2\bm{\tau}_\gamma)_\gamma  - (w,\mu)_\gamma + (w,\mu)_\gamma + \overline{\epsilon} (\mu,\mu)_{1/2,\gamma}\\
        &= \rho_f ||\bm{v}||_{0,\Omega_f}^2 + 2\nu_f \Delta t ||D(\bm{v})||_{0,\Omega_f}^2 + \rho_p ||\bm{\varphi}||_{0,\Omega_p}^2 
     + 2\nu_p \Delta t^2 ||D(\bm{\varphi})||_{0,\Omega_p}^2 
         \\ & \hspace{3mm} +  \Delta t^2 \lambda ||\nabla \cdot \bm{\varphi}||_{0,\Omega_p}^2    + \frac{s_0}{\Delta t^2} ||w||_{0,\Omega_p}^2  + \frac{\kappa}{\Delta t} ||\nabla w||_{0,\Omega_p}^2 + \frac{1}{\beta \Delta t}||s_2||_{0,\gamma}^2 + \overline{\epsilon} ||\mu||_{1/2,\gamma}^2 \\
        &\geq \min\{\rho_f, 2\nu_f \Delta t\} ||\bm{v}||^2_{1,\Omega_f} + \min\{\rho_p, 2\nu_p \Delta t^2\} ||\bm{\varphi}||_{1,\Omega_p}^2 \\
        &\hspace{3mm} + \min\Big\{\frac{s_0}{\Delta t^2},\frac{\kappa}{\Delta t}\Big\} ||w||_{1,\Omega_p}^2 + \frac{1}{\beta \Delta t} ||s_2||_{0,\gamma}^2 + \overline{\epsilon} ||\mu||_{1/2,\gamma}^2 \\
         &\geq \alpha_1 || \bm{m}||_M^2
    \end{align*}
    where the inequality assumes that all parameters are strictly positive, and \\
    $\alpha_1 := \min\big\{\rho_f, 2\nu_f \Delta t, \rho_p, 2\nu_p \Delta t^2, \dfrac{s_0}{\Delta t^2}, \dfrac{\kappa}{\Delta t}, \dfrac{1}{\beta \Delta t}, \overline{\epsilon} \big\}$.
\end{proof}

\begin{theorem}
    There exists a unique solution to \eqref{FPSI_FEM:MF4toUse}. 
\end{theorem}
\begin{proof}
Combining the results of Lemma \ref{FPSI_FEM:lem:Binfsup}, which showed the inf-sup condition for $b_{MZ}(\cdot,\cdot)$ between $M$ and $Z$, and Lemma \ref{FPSI_FEM:lem:AM_ell}, which showed the coercivity of $a_M(\cdot,\cdot)$ with the added stabilization term $\overline{\epsilon} (\lambda_p^{n+1},\mu)_{1/2,\gamma}$ , we see that the saddle point system \eqref{FPSI_FEM:MF4toUse} has a unique solution.  
\end{proof}

	\section{Fully Discrete Model}\label{FPSI_FEM:sec:FullyDiscreteModelUpdates}

Next, we turn to the analysis of the fully discretized form of \eqref{FPSI_FEM:WF:allTimeDisc}. We assume here that $\Omega_f, \Omega_p$ are convex polytopal domains. Let $h_1, h_2, \text{and } h_\gamma$ represent the mesh sizes of a quasi-uniform partition of $\Omega_f, \Omega_p, \text{and } \gamma$. The conforming discrete finite element spaces are denoted by $U^{h_1} \subset U$, $Q_f^{h_1} \subset Q_f$, $X^{h_2} \subset X$, $Q_p^{h_2} \subset Q_p$,  $\Lambda_{g2}^{h_\gamma} \subset \Lambda_{g2} $, $\Lambda_\lambda^{h_\gamma} \subset \Lambda_\lambda$, and $\Lambda_{g1}^{h_\gamma} \subset \Lambda_{g1}$ for the variables $\bm{u}, p_f, \bm{\eta}, p_p, g_2$, $\lambda_p$, and $g_1$.

\noindent Define the following space $U^{h_1}_0 := \{ \bm{u}_h \in U^{h_1} \ \big| \ \bm{u}_h \big|_\gamma = 0 \}$. We assume that $U^{h_1}_0, Q_f^{h_1}$ satisfy the discrete inf-sup condition for the traditional Stokes problem:
\begin{equation}\label{FPSI_FEM:eq:InfSup_Uh_Qh}
\underset{0\neq q_h \in Q_f^{h_1}}{\inf} \underset{0\neq \bm{v}_h \in U_0^{h_1}}{\sup} \frac{(\nabla \cdot \bm{v}_h, q_h)}{||\bm{v}_h||_1||q_h||_0} \geq \beta^* > 0.
\end{equation}

Posing \eqref{FPSI_FEM:WF:allTimeDisc} over the discrete finite element spaces, we group variables in the same way as in the continuous case, defining $U^{h_1} \times X^{h_2} \times Q_p^{h_2} \times \Lambda_{g2}^{h_\gamma} := Y^h \subset Y$, and $Y^h \times \Lambda_\lambda^{h_\gamma} := M^h \subset M$ with norms inherited from $Y$ and $M$. Likewise $Q_f^{h_1} \times \Lambda_{g1}^{h_\gamma} := Z^h \subset Z$, with the $Z$ norm retained as well. This yields the discrete equivalence of  \eqref{FPSI_FEM:MF4toUse}:

\noindent \textit{Find $\bm{m}_h^{n+1} := (\bm{u}_h^{n+1},\widehat{\bm{\eta}}_h^{\ n+1},\widehat{p}_{p,h}^{\ n+1},\widehat{g}_{2,h}^{\ n+1}, \lambda_{p,h}^{n+1}) \in M^h$, $ \bm{z}_h^{n+1} := (\widehat{p}_{f,h}^{\ n+1},\widehat{g}_{1,h}^{\ n+1}) \in Z^h$ such that:}
\begin{align}\label{FPSI_FEM:discreteSaddlePoint}
    \begin{split}
        a_M\left( \bm{m}_h^{n+1}, \bm{\zeta}_h \right) +b_{MZ}(\bm{\zeta}_h,\bm{z}_h^{n+1}) 
        &= \mathcal{F}_M(\bm{\zeta}_h) \hspace{5mm} \forall \hspace{1mm} \bm{\zeta}_h := (\bm{v}_h,\bm{\varphi}_h,w_h,s_{2,h},\mu_h) \in M^h,\\
b_{MZ}(\bm{m}_h^{n+1},\bm{\xi}_h) &= -\mathcal{F}_3(q_h,s_{1,h}) \hspace{5mm} \forall \hspace{1mm} \bm{\xi}_h := (q_h,s_{1,h}) \in Z^h .\\
    \end{split}
\end{align}

To show the well posedness of \eqref{FPSI_FEM:discreteSaddlePoint}, we will show the inf-sup condition between $M^h$ and $Z^h$. The coercivity of $a_M(\cdot,\cdot)$ on $M^h$ is inherited from the coercivity on $M$.

\begin{theorem}\label{FPSI_FEM:thm:DiscInfSup}
There exists a positive constant $\beta_2$ such that
\begin{align}\label{FPSI_FEM:infsup-disc}
\underset{0 \neq \bm{\zeta}_h \in M^h}{\sup} \frac{b_{MZ}(\bm{\zeta}_h,\bm{\xi}_h)}{||\bm{\zeta}_h||_{M}} \geq \beta_2 ||\bm{\xi}_h||_{Z} > 0 \hspace{3mm} \forall \hspace{1mm} \bm{\xi}_h \in Z^h.
\end{align}
\end{theorem}

\begin{proof}
    Let $\bm{\xi}_h := (q_h, s_{1,h}) \in Z^h$ be given. We begin by showing an inf-sup condition between $U^{h_1}$ and $\Lambda_{g1}^{h_\gamma}$. For $\bm{k} \in \bm{H^{1/2}}(\gamma)$, we may find $\bm{v}_k \in U$ s.t. 
\begin{align}\label{vkConds}
\begin{split}
    \bm{v}_k \big|_\gamma &= \bm{k}\\
||\bm{v}_k||_1 &\leq C_1 || \bm{k}||_{1/2,\gamma},
\end{split}
\end{align}
where $C_1$ represents the constant from the lifting operator from $H^{1/2}(\gamma)$ to $H^1(\Omega_f).$ Since $s_{1,h} \in \Lambda_{g1}^{h_\gamma}$, and thus $s_{1,h} \bm{n}_f \in \bm{H^{-1/2}}(\gamma)$, by \eqref{vkConds} and the definition of the dual norm,
\begin{align}\label{sh_bound}
\begin{split}
    ||s_{1,h} \bm{n}_f||_{-1/2,\gamma} &= \underset{\bm{k} \in \bm{H^{1/2}}(\gamma)}{\sup} \frac{\langle s_{1,h} \bm{n}_f,\bm{k}\rangle_\gamma}{||\bm{k}||_{1/2,\gamma}} \leq \underset{\bm{k} \in \bm{H^{1/2}}(\gamma)}{\sup} \frac{C_1 \langle s_{1,h} \bm{n}_f,\bm{v}_k\rangle_\gamma}{||\bm{v}_k||_{1}} \\
    &\leq  \underset{\bm{v}_k \in U}{\sup} \frac{C_1 \langle s_{1,h} \bm{n}_f,\bm{v}_k\rangle_\gamma}{||\bm{v}_k||_{1}} .
    \end{split}
\end{align}
Now, we note that there exists an interpolant $\mathcal{I}_f^h : U \rightarrow U^h$ which satisfies the following two conditions for $\bm{v}_k \in U$ and $\bm{v}_k^h \in U^h$ (\cite{Boffi_2013}, Section 2.5): 
\begin{align}\label{ProjectorConds}
\langle s_{1,h} \bm{n}_f, \bm{v}_k\rangle_\gamma &= \langle s_{1,h} \bm{n}_f, \mathcal{I}_f^h\bm{v}_k\rangle_\gamma, \quad \forall \ s_{1,h} \in \Lambda_{g1}^{h_\gamma} \\
||\mathcal{I}_f^h \bm{v}_k||_1 &\leq C_2 || \bm{v}_k||_1.
\end{align}
The interpolant requires $\Lambda_{g1}^{h_\gamma} \subset L^2$ on each element; the use of $P_k$ polynomials for this LM space satisfies this condition.
\noindent With these properties, we have for a given $\bm{v}_k \in U$ and $s_{1,h}  \in \Lambda_{g1}^{h_\gamma}$
\begin{align*}
 \frac{C_1 \langle s_{1,h} \bm{n}_f,\bm{v}_k\rangle_\gamma}{||\bm{v}_k||_{1}}  &=   \frac{C_1 \langle s_{1,h} \bm{n}_f,\mathcal{I}_f^h\bm{v}_k\rangle_\gamma}{||\bm{v}_k||_{1}}  \leq  \frac{C_1 C_2 \langle s_{1,h} \bm{n}_f,\mathcal{I}_f^h\bm{v}_k\rangle_\gamma}{||\mathcal{I}_f^h\bm{v}_k||_{1}} \leq \underset{\bm{v}_k^h \in U^h}{\sup}\frac{C_1 C_2 \langle s_{1,h} \bm{n}_f,\bm{v}_k^h\rangle_\gamma}{||\bm{v}_k^h||_{1}}.
\end{align*}
Taking the supremum over $\bm{v}_k \in U$ of this inequality yields:
\begin{align*}
\underset{\bm{v}_k \in U}{\sup}  \frac{C_1 \langle s_{1,h} \bm{n}_f,\bm{v}_k\rangle_\gamma}{||\bm{v}_k||_{1}} \leq \underset{\bm{v}_k \in U}{\sup} \underset{\bm{v}_k^h \in U^h}{\sup}\frac{C_1 C_2 \langle s_{1,h} \bm{n}_f,\bm{v}_k^h\rangle_\gamma}{||\bm{v}_k^h||_{1}} = \underset{\bm{v}_k^h \in U^h}{\sup}\frac{C_1 C_2 \langle s_{1,h} \bm{n}_f,\bm{v}_k^h\rangle_\gamma}{||\bm{v}_k^h||_{1}}.
\end{align*}

\noindent Thus, combining \eqref{sh_bound} with the above, we have the desired inf-sup condition:
\begin{align*}
||s_{1,h} \bm{n}_f||_{-1/2,\gamma} \leq  \underset{\bm{v}_k \in  U}{\sup} \frac{C_1 \langle s_{1,h} \bm{n}_f,\bm{v}_k\rangle_\gamma}{||\bm{v}_k||_{1}} \leq C_1 C_2\underset{\bm{v}_k^h \in U^h}{\sup}\frac{ \langle s_{1,h} \bm{n}_f,\bm{v}_k^h\rangle_\gamma}{||\bm{v}_k^h||_{1}}.
\end{align*}
Rewritten,
\begin{align}\label{infsup_LM_Stokes}
\underset{\bm{v}_k^h \in U^h}{\sup}\frac{ \langle s_{1,h} \bm{n}_f,\bm{v}_k^h\rangle_\gamma}{||\bm{v}_k^h||_{1}} \geq \beta_f ||s_{1,h} \bm{n}_f||_{-1/2,\gamma} \quad \forall s_{1,h}  \in \Lambda_{g1}^{h_\gamma}.
\end{align}

\noindent Similarly, we may show the inf-sup condition between a subspace of $X^{h_2}$ and $\Lambda_{g1}^{h_\gamma}$. For $\bm{k} \in \bm{H^{1/2}}(\gamma)$, we may find $\bm{\varphi}_k \in X$ such that $\bm{\varphi}_k\big|_\gamma = \bm{k}$ and $(\bm{\varphi}_k\cdot \bm{\tau}_\gamma)\big|_\gamma = 0$. I.e., 
\begin{align}\label{phikConds}
\begin{split}
\bm{\varphi}_k\big|_\gamma &= (\bm{\varphi}_k \cdot \bm{n}_p)\big|_\gamma \bm{n}_p + (\bm{\varphi}_k \cdot \bm{\tau}_\gamma)\big|_\gamma \bm{\tau}_\gamma = (\bm{\varphi}_k \cdot \bm{n}_p)\big|_\gamma \bm{n}_p = \bm{k}\\
||\bm{\varphi}_k||_1 &\leq C_3 ||\bm{k}||_{1/2,\gamma}.
\end{split}
\end{align}

\noindent Let $s_{1,h} \in \Lambda_{g1}^{h_\gamma}$, then by \eqref{phikConds} and the definition of the dual norm,
\begin{align}\label{sh_bound_phi}
\begin{split}
    ||s_{1,h} \bm{n}_p||_{-1/2,\gamma} &= \underset{\bm{k} \in \bm{H^{1/2}}(\gamma)}{\sup} \frac{\langle s_{1,h} \bm{n}_p,\bm{k}\rangle_\gamma}{||\bm{k}||_{1/2,\gamma}} \leq \underset{\bm{k} \in \bm{H^{1/2}}(\gamma)}{\sup} \frac{C_3 \langle s_{1,h} \bm{n}_p,\bm{\varphi}_k\rangle_\gamma}{||\bm{\varphi}_k||_{1}} \\
    &\leq  \underset{\bm{\varphi}_k \in X}{\sup} \frac{C_3 \langle s_{1,h} \bm{n}_p,\bm{\varphi}_k\rangle_\gamma}{||\bm{\varphi}_k||_{1}} .
    \end{split}
\end{align}

\noindent Define the following subspace of $X^{h_2}$:
\begin{equation}
    X_\tau^h := \{ \bm{\varphi}^h \in X^{h_2} \ \Big| \ \bm{\varphi}^h \cdot \bm{\tau} = 0 \ \text{ on } \gamma \}.
\end{equation}
The interpolant $\mathcal{I}_s^h : X \rightarrow X^h$ satisfies the following conditions for $\bm{\varphi}_k \in X$ and $\bm{\varphi}_k^h \in X_\tau^h$ (\cite{Boffi_2013}): 
\begin{align}\label{ProjectorConds_phi}
\langle s_{1,h} \bm{n}_p, \bm{\varphi}_k\rangle_\gamma &= \langle s_{1,h} \bm{n}_p, \mathcal{I}_s^h\bm{\varphi}_k\rangle_\gamma, \quad \forall \ s_{1,h} \in \Lambda_{g1}^{h_\gamma} \\
||\mathcal{I}_s^h \bm{\varphi}_k||_1 &\leq C_4 || \bm{\varphi}_k||_1.
\end{align}

\noindent With these two properties, we have for a given $\bm{\varphi}_k \in X$ and $s_{1,h} \in \Lambda_{g1}^{h_\gamma}$
\begin{align*}
 \frac{C_3 \langle s_{1,h} \bm{n}_p,\bm{\varphi}_k\rangle_\gamma}{||\bm{\varphi}_k||_{1}}  &=   \frac{C_3 \langle s_{1,h} \bm{n}_p,\mathcal{I}_s^h\bm{\varphi}_k\rangle_\gamma}{||\bm{\varphi}_k||_{1}}  \leq  \frac{C_3 C_4 \langle s_{1,h} \bm{n}_p,\mathcal{I}_s^h\bm{\varphi}_k\rangle_\gamma}{||\mathcal{I}_s^h\bm{\varphi}_k||_{1}} \\
 &\leq \underset{\bm{\varphi}_k^h \in X_\tau^h}{\sup}\frac{C_3 C_4 \langle s_{1,h} \bm{n}_p,\bm{\varphi}_k^h\rangle_\gamma}{||\bm{\varphi}_k^h||_{1}}.
\end{align*}
Taking the supremum over $\bm{\varphi}_k \in X$ of this inequality yields:
\begin{align*}
\underset{\bm{\varphi}_k \in X}{\sup}  \frac{C_3 \langle s_{1,h} \bm{n}_p,\bm{\varphi}_k\rangle_\gamma}{||\bm{\varphi}_k||_{1}} \leq \underset{\bm{\varphi}_k \in X}{\sup} \underset{\bm{\varphi}_k^h \in X_\tau^h}{\sup}\frac{C_3 C_4 \langle s_{1,h} \bm{n}_p,\bm{\varphi}_k^h\rangle_\gamma}{||\bm{\varphi}_k^h||_{1}} = \underset{\bm{\varphi}_k^h \in X_\tau^h}{\sup}\frac{C_3 C_4 \langle s_{1,h} \bm{n}_p,\bm{\varphi}_k^h\rangle_\gamma}{||\bm{\varphi}_k^h||_{1}}.
\end{align*}

\noindent Thus, combining with \eqref{sh_bound_phi}, we have the desired inf-sup condition:
\begin{align*}
||s_{1,h} \bm{n}_p||_{-1/2,\gamma} \leq  \underset{\bm{\varphi}_k \in  X}{\sup} \frac{C_3 \langle s_{1,h} \bm{n}_p,\bm{\varphi}_k\rangle_\gamma}{||\bm{\varphi}_k||_{1}} \leq C_3 C_4\underset{\bm{\varphi}_k^h \in X_\tau^h}{\sup}\frac{ \langle s_{1,h} \bm{n}_p,\bm{\varphi}_k^h\rangle_\gamma}{||\bm{\varphi}_k^h||_{1}}.
\end{align*}
Rewritten, 
\begin{align}\label{infsup_LM_Structure}
 \underset{\bm{\varphi}_k^h \in X_\tau^h}{\sup}\frac{ \langle s_{1,h} \bm{n}_p,\bm{\varphi}_k^h\rangle_\gamma}{||\bm{\varphi}_k^h||_{1}} \geq  \beta_s||s_{1,h} \bm{n}_p||_{-1/2,\gamma} \quad \forall s_{1,h}  \in \Lambda_{g1}^{h_\gamma}.
\end{align}
Now, we combine the inf-sup conditions for each piece to show the inf-sup condition for the term $b_{MZ}(\cdot,\cdot)$.  By \eqref{infsup_LM_Stokes}, there exists $\bm{u}_{1,h} \in U^{h_1}$ such that 
\begin{align}\label{V2_Prop}
||\bm{u}_{1,h}||_1 = 1 \quad \text{ and } \quad -\langle \bm{u}_{1,h},s_{1,h} \bm{n}_f\rangle_\gamma \geq \beta_f ||s_{1,h} \bm{n}_f||_{-1/2,\gamma} =\beta_f ||s_{1,h}||_{-1/2,\gamma}.
\end{align}

\noindent Likewise, by \eqref{FPSI_FEM:eq:InfSup_Uh_Qh}, there exists $\bm{u}_{2,h} \in U_0^{h_1}$ with
\begin{align}\label{V1_Prop}
||\bm{u}_{2,h}||_1 = 1 \quad \text{ and } \quad -(\nabla \cdot \bm{u}_{2,h},q_h) \geq \beta^* ||q_h||_{0}.
\end{align}

\noindent Lastly, by \eqref{infsup_LM_Structure}, there exists $\bm{\varphi}_h \in X_\tau^h$ such that  
\begin{align}\label{Varphi_Prop}
||\bm{\varphi}_h||_1 = 1 \quad \text{ and } \quad \langle \bm{\varphi}_h,s_{1,h} \bm{n}_p\rangle_\gamma \geq \beta_s ||s_{1,h} \bm{n}_p||_{-1/2,\gamma} =\beta_s ||s_{1,h}||_{-1/2,\gamma}.
\end{align}

\noindent By the Cauchy-Schwarz inequality and Korn's inequality \eqref{FPSI_FEM:eq:TraceIneqs},
\begin{equation}\label{UseCS}
\big|(\nabla \cdot \bm{u}_{1,h},q_h)\big| \leq ||\nabla \cdot \bm{u}_{1,h}||_{0} \ ||q_h||_{0} \leq \sqrt{2}C_K ||D( \bm{u}_{1,h})||_0 \ ||q_h||_0 \leq  \sqrt{2}C_K ||q_h||_0.
\end{equation}
Define $\mu_h \in \Lambda_\lambda^h$ as $\mu_h := (\bm{\varphi}_h \cdot \bm{n}_p)\big|_\gamma .$
Then, recalling that $\bm{\varphi}^h \in X_\tau^h$ implies that $(\bm{\varphi}_h \cdot \bm{\tau})\big|_\gamma = 0$,
\begin{align*}
||\mu_h||_{1/2,\gamma} = ||\mu_h \bm{n}_p||_{1/2,\gamma} = ||(\bm{\varphi}_h \cdot \bm{n}_p)\bm{n}_p||_{1/2,\gamma} = \underset{\underset{\bm{\eta}|_\gamma= (\bm{\varphi}_h\cdot \bm{n}_p) \bm{n}_p}{\bm{\eta}\in X } }{\inf} ||\bm{\eta}||_1 \leq || \bm{\varphi}_h||_1 = 1.
\end{align*}

\noindent Define $\widehat{C}:=(1 + \dfrac{\sqrt{2}C_K}{\beta^*})$. With $\bm{u}_h := \bm{u}_{1,h} + \widehat{C} \bm{u}_{2,h} \in U^{h_1}$, let the element $(\bm{u}_h,\bm{\varphi}_h,0,0,\mu_h) = \bm{m}_h \in M^h$. Applying \eqref{V2_Prop}-\eqref{UseCS} in the following yields
\begin{align*}
&b_{MZ}(\bm{m}_h, (q_h,s_{1,h})) = -(\nabla \cdot \bm{u}_h, q_h) - \langle \bm{u}_h,s_{1,h} \bm{n}_f\rangle_\gamma - \langle \bm{\varphi}_h,s_{1,h} \bm{n}_p\rangle_\gamma + \langle s_{1,h}, \mu_h \rangle_\gamma \\
&= -(\nabla \cdot \bm{u}_{1,h}, q_h) -\widehat{C} (\nabla \cdot \bm{u}_{2,h}, q_h) - \langle \bm{u}_{1,h},s_{1,h} \bm{n}_f\rangle_\gamma - \widehat{C} \langle \bm{u}_{2,h},s_{1,h} \bm{n}_f\rangle_\gamma \\
&\geq -\sqrt{2}C_K ||q_h||_0 + \widehat{C}\beta^*||q_h||_0+ \beta_f ||s_{1,h} ||_{-1/2,\gamma}\\
&\geq \min\{ \beta^*, \beta_f\} (||q_h||_0 + ||s_{1,h} ||_{-1/2,\gamma}).
\end{align*}

\noindent Note that 
\begin{align*}
||\bm{u}_h||_1 + ||\bm{\varphi}_h||_1 + || \mu_h ||_{1/2,\gamma} \leq ||\bm{u}_{1,h}||_1 + \widehat{C} || \bm{u}_{2,h}||_1 + || \bm{\varphi}_h||_1 + ||\mu_h||_{1/2,\gamma} \leq 4 +  \frac{\sqrt{2}C_K}{\beta^*} ,
\end{align*}
which implies  $$ \frac{||\bm{u}_h||_1 + ||\bm{\varphi}_h||_1 + ||\mu_h||_{1/2,\gamma}}{4 +  \frac{\sqrt{2}C_K}{\beta^*}} \leq 1.$$

\noindent Thus, we have that
\begin{align*}
&\underset{0 \neq \bm{\zeta}_h \in M^h}{\sup} \frac{b_{MZ}(\bm{\zeta}_h,(q_h,s_{1,h}))}{||\bm{\zeta}_h||_{M}} \geq \frac{b_{MZ}(\bm{m}_h, (q_h,s_{1,h}))}{||\bm{m}_h||_M} \geq \min\{\beta^*, \beta_f\} \frac{ ||q_h||_0 + ||s_{1,h} ||_{-1/2,\gamma} }{||\bm{m}_h||_M} \\
&\geq   \frac{\min\{\beta^*, \beta_f\}}{4 +   \frac{\sqrt{2}C_K}{\beta^*}}\frac{ \left(||\bm{u}_h||_1 + ||\bm{\varphi}_h||_1 + ||\mu_h||_{1/2,\gamma}\right) \left( ||q_h||_0 + ||s_{1,h} ||_{-1/2,\gamma} \right)}{||\bm{m}_h||_M}\\
&\geq  \frac{\min\{\beta^*, \beta_f\}}{4 +   \frac{\sqrt{2}C_K}{\beta^*}} \frac{\left(||\bm{u}_h||_1^2 + ||\bm{\varphi}_h||_1^2 + ||\mu_h||_{1/2,\gamma}^2\right)^{1/2} \left( ||q_h||_0^2 + ||s_{1,h} ||_{-1/2,\gamma}^2 \right)^{1/2}}{\left(||\bm{u}_h||_1^2 + ||\bm{\varphi}_h||_1^2 + ||\mu_h||_{1/2,\gamma}^2\right)^{1/2}}\\
&= \frac{\min\{\beta^*, \beta_f\}}{4 +   \frac{\sqrt{2}C_K}{\beta^*}}  ||(q_h,s_{1,h} )||_Z.
\end{align*}

\end{proof}
\noindent With the discrete inf-sup condition proved in Theorem \ref{FPSI_FEM:thm:DiscInfSup}, the inherited coercivity of $a_M(\cdot, \cdot)$ on $M^h$ results in the well-posedness of the fully discrete saddle point system \eqref{FPSI_FEM:discreteSaddlePoint} (\cite{Brezzi_1990}).
\begin{theorem} \label{FPSI_FEM:fully-discrete-well-posedness}
The fully discrete system \eqref{FPSI_FEM:discreteSaddlePoint} has a unique solution\\ $\Big((\bm{u}_h^{n+1},\widehat{\bm{\eta}}_h^{\ n+1},\widehat{p}_{p,h}^{\ n+1},\widehat{g}_{2,h}^{\ n+1}, \lambda_{p,h}^{n+1}),(\widehat{p}_{f,h}^{\ n+1},\widehat{g}_{1,h}^{\ n+1})\Big) \in M^h \times Z^h.$
\end{theorem}
\noindent Next, we examine the stability and convergence properties of the formulation.

\section{Stability Analysis}\label{FPSI_FEM:sec:stability}
We consider stability of the fully discrete formulation; however, stability of the semi-discrete problem \eqref{FPSI_FEM:WF:allTimeDisc} could be similarly demonstrated as the techniques and inequalities used are not dependent upon properties of the discrete spaces. We assume homogeneous Neumann conditions $\bm{u}_N = 0, \bm{\eta}_N = 0, p_N = 0$ to simplify the problem setting, and the extension to the non-homogeneous case could be easily handled.  We also define an energy norm for the displacement
$$ ||\bm{\eta}^{n+1}||_E^2 := 2\nu_p ||D(\bm{\eta}^{n+1})||^2_{0,\Omega_p} + \lambda ||\nabla \cdot \bm{\eta}^{n+1}||_{0,\Omega_p}^2,$$
and list the following inequalities for all $\bm{u} \in U, \bm{\eta} \in X$, and all $p_p \in Q_p$ (see (\cite{Bukac_2015OpSplit}), for example). The constants $C_T, C_K, C_P$ depend only on the domains. Instead of defining individual constants for each inequality, we may find a single constant for each type of inequality (trace, Korn's, and Poincar\'{e}) to simplify notation.
\begin{align}\label{FPSI_FEM:eq:TraceIneqs} 
\begin{split}
&\text{Trace Inequ. } \quad  ||\bm{u}||_{0,\gamma} \leq C_T ||\bm{u}||_{1,\Omega_f}, \quad ||\bm{\eta}||_{0,\gamma} \leq C_T ||\bm{\eta}||_{1,\Omega_p}, \quad ||p_p||_{0,\gamma} \leq C_T ||p_p||_{1,\Omega_p} \\
&\text{Korn's Inequ. } \quad ||\nabla \bm{u}||_{0,\Omega_f} \leq C_K ||D(\bm{u})||_{0,\Omega_f}, \quad ||\nabla \bm{\eta}||_{0,\Omega_p} \leq C_K ||D(\bm{\eta})||_{0,\Omega_p}. 
\end{split}
\end{align}
As functions in $U, X, Q_p$ satisfy homogeneous Dirichlet conditions on a component of the boundary, the Poincar\'{e} inequalities hold:
\begin{align}\label{FPSI_FEM:eq:PoincareIneqs}
\begin{split}
||\bm{u}||_{0,\Omega_f} \leq C_P ||\nabla \bm{u}||_{0,\Omega_f}, \quad ||\bm{\eta}||_{0,\Omega_p} \leq C_P ||\nabla \bm{\eta}||_{0,\Omega_p}, \quad ||p_p||_{0,\Omega_p} \leq C_P ||\nabla p_p||_{0,\Omega_p}.
\end{split}
\end{align}
Combining Korn's and Poincar\'{e} inequalities, we see that there exists a constant $C_{KP}$ satisfying
\begin{equation}\label{FPSI_FEM:eq:1normbound}
||\bm{u}||_{1,\Omega_f}^2 \leq C_{KP} ||D(\bm{u})||^2_{0,\Omega_f}, \quad ||\bm{\eta}||_{1,\Omega_p}^2 \leq  C_{KP} ||D(\bm{\eta})||_{0,\Omega_p}^2, \quad ||p_p||_{1,\Omega_p}^2 \leq C_{KP} ||\nabla p_p||_{0,\Omega_p}^2.
\end{equation}
We list the remaining constants that will be utilized throughout the proof: 
\begin{align}\label{FPSI_FEM:eq:StabConstants}
    \begin{split}
    C_\eta &:= \frac{C_{KP}}{2\nu_p}, \quad C_1 := C_T \sqrt{C_{KP}}, \quad 
          C_2 := \sqrt{d} C_K C_P.
    \end{split}
\end{align}
The weak form \eqref{FPSI_FEM:WF:allTimeDisc}
is restated over the discrete spaces, moving inner products with previous time step terms to the left hand side and scaling by $\Delta t$ to obtain the finite difference approximations for the derivative terms defined in \eqref{FPSI_FEM:eq:TimeDerivNotation}. We also add the stabilization term  $\bar{\epsilon} \hspace{0.2mm} a_\lambda(\cdot,\cdot)$ for consistency with the SP formulation used for the well-posedness results. Thus for all $\bm{v}_h \in U^{h_1}, \bm{\varphi}_h \in X^{h_2}$, $w_h \in Q^{h_2}_p, s_{w,h} \in \Lambda_{g2}^{h_\gamma}$, $\mu_h \in \Lambda_\lambda^{h_\gamma}, q_h \in Q_f^{h_1}$, and $s_{1,h} \in \Lambda_{g1}^{h_\gamma}$,
\begin{align}\label{FPSI_FEM:WF:allTimeDisc_StabError_spaceDisc}
    \begin{split}
       &\rho_f ( \bm{\dot{u}}_h^{n+1},  \bm{v}_h )_{\Omega_f} + 2  \nu_f \left( D( \bm{u}_h^{n+1}), D( \bm{v}_h) \right)_{\Omega_f} -  (p_{f,h}^{n+1}, \nabla \cdot  \bm{v}_h)_{\Omega_f} -  \langle g_{1,h}^{n+1} \bm{n}_f, \bm{v}_h\rangle_\gamma \\
      &\hspace{25mm}  -  ( g_{2,h}^{n+1} \bm{\tau}_\gamma, \bm{v}_h)_\gamma =  \langle\bm{f}_f^{n+1}, \bm{v}_h\rangle_{\Omega_f} \\
&\frac{\rho_p}{\Delta t} \left( \bm{\dot{\eta}}_h^{n+1} - \bm{\dot{\eta}}_h^n, \bm{\varphi}_h \right)_{\Omega_p} +  2  \nu_p \left( D\left( \bm{\eta}_h^{n+1}\right), D(\bm{\varphi}_h) \right)_{\Omega_p} + \lambda  \left( \nabla \cdot \bm{\eta}_h^{n+1}, \nabla \cdot \bm{\varphi}_h \right)_{\Omega_p} \\
& \hspace{25mm}- \alpha (p_{p,h}^{n+1}, \nabla \cdot\bm{\varphi}_h)_{\Omega_p}- \langle g_{1,h}^{n+1}\bm{n}_p,\bm{\varphi}_h\rangle_\gamma +  ( g_{2,h}^{n+1}\bm{\tau}_\gamma, \bm{\varphi}_h)_\gamma  =  \langle\bm{f}_\eta^{n+1}, \bm{\varphi}_h \rangle_{\Omega_p}   \\
&s_0 ( \dot{p}_{p,h}^{n+1}, w_h)_{\Omega_p} + \alpha\left(\nabla \cdot \bm{\dot{\eta}}_h^{n+1}, w_h \right)_{\Omega_p} + \kappa  (\nabla p_{p,h}^{n+1}, \nabla w_h)_{\Omega_p} -  ( \lambda_{p,h}^{n+1}, w_h)_{\gamma} =  \langle f_p^{n+1},w_h \rangle_{\Omega_p}  \\
  &  \frac{1}{\beta } ( g_{2,h}^{n+1},s_{2,h})_\gamma  +  ( \bm{u}_h^{n+1} \cdot \bm{\tau}_\gamma,s_{2,h})_\gamma - \left(\bm{\dot{\eta}}_h^{n+1} \cdot \bm{\tau}_\gamma,s_{2,h}\right)_\gamma = 0\\
     &  \langle g_{1,h}^{n+1},\mu_h \rangle_\gamma  + (  p_{p,h}^{n+1},\mu_h)_\gamma + \overline{\epsilon} (\lambda_{p,h}^{n+1},\mu_h)_{1/2,\gamma} = 0 \\
     &( \nabla \cdot  \bm{u}_h^{n+1}, q_h )_{\Omega_f} = 0  \\
  & \langle  \bm{u}_h^{n+1} \cdot \bm{n}_f,s_{1,h}\rangle_\gamma + \left\langle \bm{\dot{\eta}}_h^{n+1} \cdot \bm{n}_p ,s_{1,h} \right\rangle_\gamma   -  \langle \lambda_{p,h}^{n+1},s_{1,h}\rangle_\gamma = 0.
    \end{split}
\end{align}

\noindent The stability of this system is stated in the following theorem.
\begin{theorem}\label{FPSI_FEM:thm:StabilitySB}
Assume that $\bm{f}_f \in L^2(0,T;\bm{H^{-1}}(\Omega_f))$, $\bm{f}_\eta \in H^1(0,T;\bm{H^{-1}}(\Omega_p))$, and $f_p \in L^2(0,T;H^{-1}(\Omega_p))$. Let $\Big\{\big((\bm{u}^n_h,\bm{\eta}^n_h,p^n_{p,h},g^n_{2,h},\lambda^n_{p,h}),$ $(p^n_{f,h},g^n_{1,h})\big)\Big\}_{0\leq n \leq N} \in M^h \times Z^h$ be the solution of \eqref{FPSI_FEM:WF:allTimeDisc_StabError_spaceDisc}, with $N$ the total number of time steps; i.e. $T = N \Delta t.$ Then the following estimate holds for $\overline{C}, C^*, \varepsilon_1$ defined in \eqref{FPSI_FEM:eq:eps1defn},\eqref{FPSI_FEM:eq:Cfinaldefns}: 
\begin{small}
\begin{align*}
    \begin{split}
       & ||\bm{\eta}^N_h ||_E^2 +   ||  \bm{u}_h^{N} ||^2_{0}      +  || \bm{\dot{\eta}}_h^{N}||^2_{0}    +  || p_{p,h}^{N}||^2_{0}  + \Delta t \sum_{n=0}^{N} \Big[ || \bm{\ddot{\eta}}_h^{n}||_0^2 +   ||  \bm{\dot{u}}_h^{n} ||^2_{0}  +  || \bm{\dot}{p}_{p,h}^{n}||^2_{0} +  || \bm{\dot{\eta}}_h^{n} ||_E^2 \\
       &+    || D( \bm{u}_h^{n})||^2_{0}  +   || \nabla p_{p,h}^{n}||^2_{0} + || g_{2,h}^{n}||^2_{0,\gamma}  +  ||\lambda_{p,h}^{n}||^2_{1/2,\gamma}   +   ||p_{f,h}^{n}||_{0}^2 + ||g_{1,h}^{n}||_{-1/2,\gamma}^2 \Big] \\
       &\leq  \frac{C^*}{\overline{C}} \Bigg( \left(1+\frac{1}{N}\right)\Big[ \rho_f  ||  \bm{u}_h^{0} ||^2_{0} +  \rho_p || \bm{\dot{\eta}}_h^{0}||^2_{0} + s_0  || p_{p,h}^{0}||^2_{0} +   (1+2C_\eta) ||\bm{\eta}_h^0||_E^2  +  || \bm{f}_\eta^0||_{-1}^2 +  \frac{1}{\Delta t^2 \varepsilon_1} || \bm{f}_\eta^N||_{-1}^2 \Big]\\
       &+\rho_p \Delta t^2 ||\bm{\ddot{\eta}}_h^0||_0^2 + \rho_f \Delta t^2 ||\bm{\dot{u}}_h^0||_0^2 + s_0 \Delta t^2 ||\bm{\dot}{p}_{p,h}^0||_0^2 + \Delta t^2 ||\bm{\dot{\eta}}_h^0||_E^2 + 2\nu_f\Delta t ||D(\bm{u}_h^0)||_0^2 + \kappa \Delta t||\nabla p_{p,h}^0||_0^2 \\
    & + \frac{2\Delta t}{\beta} ||g_{2,h}^0||^2_{0,\gamma} + 2\bar{\epsilon}\Delta t ||\lambda_{p,h}^0||_{1/2,\gamma}^2 + \varepsilon_1 \Delta t^2 \beta_2^2 (||p_{f,h}^0||_0^2 + ||g_{1,h}^0||_{-1/2,\gamma}^2) +  \left(\Delta t+\frac{\Delta t}{N}\right)\sum_{n=0}^{N}  \Bigg[   \Delta t || \bm{\dot{f}_\eta}^{n} ||_{-1}^2   \\
    &+ 70\varepsilon_1 \Delta t  ||\bm{f}_\eta^{n}||^2_{-1} + \left( \frac{C_{KP}}{2\nu_f} + 70\varepsilon_1 \Delta t \right) ||\bm{f}_f^{n}||_{-1}^2 +  \left( \frac{C_{KP}}{\kappa} + 70\varepsilon_1 \Delta t \right)|| f_p^{n}||^2_{-1} \Bigg] \Bigg).
    \end{split}
\end{align*}
\end{small}
\end{theorem}

\begin{proof}
\noindent We present the proof in three steps.  \\
\textit{\textbf{Step 1: Bound the functions in $M^h$.}}\\

 Define $N$ such that $T = N\Delta t$, and choose the test functions $ \bm{v}_h =  \bm{u}_h^{n+1}$, $ \bm{\varphi}_h = \bm{\dot{\eta}}_h^{n+1}$, $w_h = p_{p,h}^{n+1} $, $s_{2,h} = g_{2,h}^{n+1}, \mu_h = \lambda_{p,h}^{n+1}$, $ q_h = p_{f,h}^{n+1},$ and $s_{1,h} = g_{1,h}^{n+1}$ in \eqref{FPSI_FEM:WF:allTimeDisc_StabError_spaceDisc} to obtain 
\begin{align*}
    \begin{split}
       &\rho_f ( \bm{\dot{u}_h}^{n+1},  \bm{u}_h^{n+1} )_{\Omega_f} + 2  \nu_f \left( D( \bm{u}_h^{n+1}), D( \bm{u}_h^{n+1}) \right)_{\Omega_f} -  (p_{f,h}^{n+1}, \nabla \cdot  \bm{u}_h^{n+1})_{\Omega_f} -  \langle g_{1,h}^{n+1} \bm{n}_f, \bm{u}_h^{n+1}\rangle_\gamma  \\
      &\hspace{15mm} -  ( g_{2,h}^{n+1} \bm{\tau}_\gamma, \bm{u}_h^{n+1})_\gamma=  \langle\bm{f}_f^{n+1}, \bm{u}_h^{n+1}\rangle_{\Omega_f}   \\
&\frac{\rho_p}{\Delta t} \left( \bm{\dot{\eta}}_h^{n+1} - \bm{\dot{\eta}_h}^n, \bm{\dot{\eta}}_h^{n+1} \right)_{\Omega_p} +  2  \nu_p \left( D\left( \bm{\eta}_h^{n+1}\right), D(\bm{\dot{\eta}}_h^{n+1}) \right)_{\Omega_p} + \lambda  \left( \nabla \cdot  \bm{\eta}_h^{n+1}, \nabla \cdot \bm{\dot{\eta}}_h^{n+1} \right)_{\Omega_p} \\
& \hspace{15mm}- \alpha (p_{p,h}^{n+1}, \nabla \cdot\bm{\dot{\eta}}_h^{n+1})_{\Omega_p}- \langle g_{1,h}^{n+1}\bm{n}_p, \bm{\dot{\eta}}_h^{n+1}\rangle_\gamma +  ( g_{2,h}^{n+1}\bm{\tau}_\gamma, \bm{\dot{\eta}}_h^{n+1})_\gamma  =  \langle\bm{f}_\eta^{n+1}, \bm{\dot{\eta}}_h^{n+1}\rangle_{\Omega_p}  \\
&s_0 ( \dot{p}_{p,h}^{n+1}, p_{p,h}^{n+1})_{\Omega_p} + \alpha\left(\nabla \cdot \bm{\dot{\eta}}_h^{n+1}, p_{p,h}^{n+1} \right)_{\Omega_p} + \kappa  (\nabla p_{p,h}^{n+1}, \nabla p_{p,h}^{n+1})_{\Omega_p} -  ( \lambda_{p,h}^{n+1}, p_{p,h}^{n+1})_{\gamma} =  \langle f_p^{n+1},p_{p,h}^{n+1} \rangle_{\Omega_p}   \\
      &  \frac{1}{\beta } ( g_{2,h}^{n+1},g_{2,h}^{n+1})_\gamma  +  ( \bm{u}_h^{n+1} \cdot \bm{\tau}_\gamma,g_{2,h}^{n+1})_\gamma - \left(\bm{\dot{\eta}}_h^{n+1} \cdot \bm{\tau}_\gamma,g_{2,h}^{n+1}\right)_\gamma = 0 \\
     &  \langle g_{1,h}^{n+1},\lambda_{p,h}^{n+1}\rangle_\gamma  + (  p_{p,h}^{n+1},\lambda_{p,h}^{n+1})_\gamma + \overline{\epsilon} (\lambda_{p,h}^{n+1},\lambda_{p,h}^{n+1})_{1/2,\gamma} = 0\\
       &( \nabla \cdot  \bm{u}_h^{n+1}, p_{f,h}^{n+1} )_{\Omega_f} = 0 \\
  & \langle  \bm{u}_h^{n+1} \cdot \bm{n}_f,g_{1,h}^{n+1}\rangle_\gamma + \langle \bm{\dot{\eta}}_h^{n+1} \cdot \bm{n}_p ,g_{1,h}^{n+1} \rangle_\gamma   -  \langle \lambda_{p,h}^{n+1},g_{1,h}^{n+1}\rangle_\gamma = 0 .
    \end{split}
\end{align*}

\noindent Notice that the sum of the last two equations results in the bilinear form $-b_{MZ}(\cdot,\cdot)$, i.e.,
\begin{align}\label{FPSI_FEM:eq:constraintsAdd}
\begin{split}
  -b_{MZ}\left( (\bm{u}_h^{n+1},\bm{\dot{\eta}}_h^{n+1},p_{p,h}^{n+1},g_{2,h}^{n+1},\lambda_{p,h}^{n+1}),(p_{f,h}^{n+1},g_{1,h}^{n+1})\right) = 0.
    \end{split}
\end{align}
Adding the first five equations causes most of the mixed terms to drop:
\begin{small}
\begin{align}
    \begin{split}
       &\rho_f ( \bm{\dot{u}}_h^{n+1},  \bm{u}_h^{n+1} )_{\Omega_f} + 2  \nu_f \left( D( \bm{u}_h^{n+1}), D( \bm{u}_h^{n+1}) \right)_{\Omega_f}   + \frac{\rho_p}{\Delta t} \left( \bm{\dot{\eta}}_h^{n+1}- \bm{\dot{\eta}}_h^n, \bm{\dot{\eta}}_h^{n+1} \right)_{\Omega_p} \\
       & +  2  \nu_p \left( D\left( \bm{\eta}_h^{n+1}\right), D(\bm{\dot{\eta}}_h^{n+1}) \right)_{\Omega_p} 
        +\lambda  \left( \nabla \cdot  \bm{\eta}_h^{n+1}, \nabla \cdot \bm{\dot{\eta}}_h^{n+1} \right)_{\Omega_p}   +s_0 ( \dot{p}_{p,h}^{n+1}, p_{p,h}^{n+1})_{\Omega_p} \\
        &+ \kappa  (\nabla p_{p,h}^{n+1}, \nabla p_{p,h}^{n+1})_{\Omega_p} + \frac{1}{\beta } ( g_{2,h}^{n+1},g_{2,h}^{n+1})_\gamma  
       + \overline{\epsilon} (\lambda_{p,h}^{n+1},\lambda_{p,h}^{n+1})_{1/2,\gamma} -  (p_{f,h}^{n+1}, \nabla \cdot  \bm{u}_h^{n+1})_{\Omega_f} -  \langle g_{1,h}^{n+1} \bm{n}_f, \bm{u}_h^{n+1}\rangle_\gamma \\
       &- \langle g_{1,h}^{n+1}\bm{n}_p, \bm{\dot{\eta}}_h^{n+1}\rangle_\gamma + \langle g_{1,h}^{n+1},\lambda_{p,h}^{n+1}\rangle_\gamma =  \langle\bm{f}_f^{n+1}, \bm{u}_h^{n+1}\rangle_{\Omega_f}  + \langle\bm{f}_\eta^{n+1}, \bm{\dot{\eta}}_h^{n+1}\rangle_{\Omega_p} +  \langle f_p^{n+1},p_{p,h}^{n+1} \rangle_{\Omega_p}.
    \end{split}
\end{align}
\end{small}
\noindent By \eqref{FPSI_FEM:eq:constraintsAdd}, the remaining mixed terms equal zero. We invoke the definition of the derivative difference quotients and apply the identity $(a-b)a = \frac{1}{2} \left( a^2 - b^2 + (a-b)^2\right)$ to the left side of the equation. For the inner products on the right, we apply Young's inequality and \eqref{FPSI_FEM:eq:1normbound}:

\begin{align*}
       &\frac{\rho_f}{2 \Delta t} \Big( ||  \bm{u}_h^{n+1} ||^2_{0} - ||  \bm{u}_h^{n} ||^2_{0}   \Big)   + \frac{\rho_p}{2 \Delta t} \Big( || \bm{\dot{\eta}}_h^{n+1}||^2_{0}  - || \bm{\dot{\eta}}_h^{n}||^2_{0} \Big)   +\frac{s_0}{2 \Delta t} \Big(  || p_{p,h}^{n+1}||^2_{0} -  || p_{p,h}^{n}||^2_{0}  \Big)\\
       &+  \frac{\nu_p}{\Delta t} \Big( ||D\left( \bm{\eta}_h^{n+1} \right)||_0^2 - ||D\left(\bm{\eta}_h^n\right)||_0^2 \Big)  +  \frac{\lambda}{2 \Delta t} \Big( ||\nabla \cdot  \bm{\eta}_h^{n+1} ||_0^2 - ||\nabla \cdot \bm{\eta}_h^n||_0^2 \Big)+ \frac{\rho_f}{2 \Delta t} ||  \bm{u}_h^{n+1}- \bm{u}_h^n ||^2_{0} \\
        & + \frac{\rho_p}{2\Delta t}  || \bm{\dot{\eta}}_h^{n+1} - \bm{\dot{\eta}}_h^{n}||^2_{0}  + \frac{s_0}{2\Delta t} || p_{p,h}^{n+1}- p_{p,h}^{n}||^2_{0} + \frac{\nu_p}{\Delta t} \Big(|| D(\bm{\eta}_h^{n+1} - \bm{\eta}_h^{n}) ||_0^2 \Big) \\
       & + \frac{\lambda}{2 \Delta t} \Big(|| \nabla \cdot (\bm{\eta}_h^{n+1} -   \bm{\eta}_h^{n}) ||_0^2 \Big)   + 2  \nu_f || D( \bm{u}_h^{n+1})||^2_{0} + \kappa  || \nabla p_{p,h}^{n+1}||^2_{0}  + \frac{1}{\beta } || g_{2,h}^{n+1}||^2_{0,\gamma} + \overline{\epsilon} ||\lambda_{p,h}^{n+1}||^2_{1/2,\gamma}  \\
      & \leq  \frac{C_{KP}}{4\nu_f} ||\bm{f}_f^{n+1}||_{-1}^2 + \nu_f || D(\bm{u}_h^{n+1})||^2_{0}   +  \langle \bm{f}_\eta^{n+1},\bm{\dot{\eta}}_h^{n+1}\rangle_{\Omega_p}   +  \frac{C_{KP}}{2 \kappa}|| f_p^{n+1}||^2_{-1} + \frac{\kappa}{2} ||\nabla p_{p,h}^{n+1}||^2_{0}.
\end{align*}

\noindent We rewrite some norms in terms of $\bm{\dot{\eta}}_h, \bm{\dot{u}}_h,$ and $\bm{\dot}{p}_{p,h}$ instead of their expanded finite difference forms, i.e., $||\bm{\eta}_h^{n+1} - \bm{\eta}_h^n||^2 = \Delta t^2 ||\bm{\dot{\eta}}_h^{n+1}||^2$. Moving terms to the left gives
\begin{align*}
    \begin{split}
       &\frac{\rho_f}{2 \Delta t} \Big( ||  \bm{u}_h^{n+1} ||^2_{0} - ||  \bm{u}_h^{n} ||^2_{0}   \Big)   + \frac{\rho_p}{2 \Delta t} \Big( || \bm{\dot{\eta}}_h^{n+1}||^2_{0}  - || \bm{\dot{\eta}}_h^{n}||^2_{0} \Big)   +\frac{s_0}{2 \Delta t} \Big(  || p_{p,h}^{n+1}||^2_{0} -  || p_{p,h}^{n}||^2_{0}  \Big) \\
       & +    \frac{1}{2\Delta t} \Big( || \bm{\eta}_h^{n+1} ||_E^2 - ||\bm{\eta}_h^n||_E^2 \Big) + \frac{\rho_f \Delta t}{2} ||  \bm{\dot{u}}_h^{n+1} ||^2_{0}+ \frac{\rho_p\Delta t}{2}  || \bm{\ddot{\eta}}_h^{n+1}||^2_{0}  + \frac{s_0 \Delta t}{2} || \bm{\dot}{p}_{p,h}^{n+1}||^2_{0} \\
       &+\frac{\Delta t}{2} || \bm{\dot{\eta}}_h^{n+1} ||_E^2  +   \nu_f || D( \bm{u}_h^{n+1})||^2_{0} + \frac{\kappa}{2}  || \nabla p_{p,h}^{n+1}||^2_{0}  + \frac{1}{\beta } || g_{2,h}^{n+1}||^2_{0,\gamma} + \overline{\epsilon} ||\lambda_{p,h}^{n+1}||^2_{1/2,\gamma} \\
      & \leq  \frac{C_{KP}}{4\nu_f} ||\bm{f}_f^{n+1}||_{-1}^2    +  \frac{C_{KP}}{2 \kappa}|| f_p^{n+1}||^2_{-1} +  \langle \bm{f}_\eta^{n+1},\bm{\dot{\eta}}_h^{n+1}\rangle_{\Omega_p}.
    \end{split}
\end{align*}

\noindent We now multiply by $2\Delta t$ and sum the resulting inequality from $n=0,\ldots,N-1$:
\begin{align}\label{FPSI_FEM:stabboundMterms}
    \begin{split}
       &\rho_f ||  \bm{u}_h^{N} ||^2_{0}      + \rho_p || \bm{\dot{\eta}}_h^{N}||^2_{0}    +s_0  || p_{p,h}^{N}||^2_{0}  + || \bm{\eta}_h^{N} ||_E^2 + \Delta t \sum_{n=0}^{N-1} \Big[  \rho_p \Delta t|| \bm{\ddot{\eta}}_h^{n+1} ||^2_{0} +  \rho_f \Delta t ||  \bm{\dot{u}}_h^{n+1} ||^2_{0}  \\
       &+ s_0 \Delta t || \bm{\dot}{p}_{p,h}^{n+1}||^2_{0} + \Delta t || \bm{\dot{\eta}}_h^{n+1} ||_E^2  +   2 \nu_f || D( \bm{u}_h^{n+1})||^2_{0} + \kappa  || \nabla p_{p,h}^{n+1}||^2_{0}  + \frac{2}{\beta } || g_{2,h}^{n+1}||^2_{0,\gamma} + 2\overline{\epsilon} ||\lambda_{p,h}^{n+1}||^2_{1/2,\gamma} \Big] \\
      & \leq   \rho_f  ||  \bm{u}_h^{0} ||^2_{0} +  \rho_p || \bm{\dot{\eta}}_h^{0}||^2_{0} + s_0  || p_{p,h}^{0}||^2_{0} +   ||\bm{\eta}_h^0||_E^2\\
      &+ \Delta t \sum_{n=0}^{N-1} \Big[ \frac{C_{KP}}{2\nu_f} ||\bm{f}_f^{n+1}||_{-1}^2    +  \frac{C_{KP}}{\kappa}|| f_p^{n+1}||^2_{-1} +  2 \langle \bm{f}_\eta^{n+1},\bm{\dot{\eta}}_h^{n+1}\rangle_{\Omega_p} \Big].
    \end{split}
\end{align}

\noindent To treat the remaining inner product on the right, integrate by parts in time using the following identity for any real numbers $a$ and $b$ (\cite{Bukac_2015OpSplit}):
\begin{equation*}
    \Delta t \sum_{n=0}^{N-1} a^{n+1} \bm{\dot}{b}^{n+1} = a^N b^N - a^0 b^0 - \Delta t \sum_{n=0}^{N-1} \bm{\dot}{a}^{n+1} b^n.
\end{equation*}

\noindent Thus, the regularity assumption $\bm{f}_\eta \in H^1(0,T;H^{-1}(\Omega_p))$ provides the bound
\begin{align}\label{eq:IBP_inTime_Youngs}
\begin{split}
   &\Delta t \sum_{n=0}^{N-1} 2\langle \bm{f}_\eta^{n+1},\bm{\dot{\eta}}^{n+1}_h \rangle_{\Omega_p} = 2\langle \bm{f}_\eta^{N},\bm{\eta}^{N}_h \rangle_{\Omega_p} - 2\langle \bm{f}_\eta^{0},\bm{\eta}^{0}_h \rangle_{\Omega_p} - \Delta t \sum_{n=0}^{N-1} 2\langle \bm{\dot{f}_\eta}^{n+1},\bm{\eta}^{n}_h \rangle_{\Omega_p} \\
   &\leq \frac{1}{\Delta t^2 \varepsilon_1}|| \bm{f}_\eta^N||_{-1}^2 +  \Delta t^2 \varepsilon_1 ||\bm{\eta}^N_h ||_1^2 +  || \bm{f}_\eta^0||_{-1}^2 +  ||\bm{\eta}^0_h||_1^2 + \Delta t \sum_{n=0}^{N-1} \Big[ \Delta t || \bm{\dot{f}_\eta}^{n+1} ||_{-1}^2 + \frac{1}{\Delta t} ||\bm{\eta}^n_h||_1^2  \Big],
   \end{split}
\end{align}
where $\varepsilon_1 > 0$ is the constant from Young's inequality, which will be defined later. Using the identity $||\bm{\eta}||_1^2 \leq C_\eta || \bm{\eta}||_E^2$, \eqref{FPSI_FEM:stabboundMterms} becomes
\begin{align}\label{FPSI_FEM:stabboundMterms2}
    \begin{split}
       &\rho_f ||  \bm{u}_h^{N} ||^2_{0}      + \rho_p || \bm{\dot{\eta}}_h^{N}||^2_{0}    +s_0  || p_{p,h}^{N}||^2_{0}  + || \bm{\eta}_h^{N} ||_E^2 + \Delta t \sum_{n=0}^{N-1} \Big[ \rho_p \Delta t|| \bm{\ddot{\eta}}_h^{n+1} ||^2_{0} + \rho_f \Delta t ||  \bm{\dot{u}}_h^{n+1} ||^2_{0}  \\
       &+ s_0 \Delta t || \bm{\dot}{p}_{p,h}^{n+1}||^2_{0} + \Delta t || \bm{\dot{\eta}}_h^{n+1} ||_E^2  +   2 \nu_f || D( \bm{u}_h^{n+1})||^2_{0} + \kappa  || \nabla p_{p,h}^{n+1}||^2_{0}  + \frac{2}{\beta } || g_{2,h}^{n+1}||^2_{0,\gamma} + 2\overline{\epsilon} ||\lambda_{p,h}^{n+1}||^2_{1/2,\gamma} \Big] \\
      & \leq   \rho_f  ||  \bm{u}_h^{0} ||^2_{0} +  \rho_p || \bm{\dot{\eta}}_h^{0}||^2_{0} + s_0  || p_{p,h}^{0}||^2_{0} +   (1+C_\eta) ||\bm{\eta}_h^0||_E^2 +  || \bm{f}_\eta^0||_{-1}^2 + \frac{1}{\Delta t^2 \varepsilon_1}  || \bm{f}_\eta^N||_{-1}^2 \\
      &+  \Delta t^2 \varepsilon_1 C_\eta||\bm{\eta}^N_h ||_E^2+\Delta t \sum_{n=0}^{N-1} \Big[ \frac{C_{KP}}{2\nu_f} ||\bm{f}_f^{n+1}||_{-1}^2    +  \frac{C_{KP}}{\kappa}|| f_p^{n+1}||^2_{-1}    + \Delta t || \bm{\dot{f}_\eta}^{n+1} ||_{-1}^2 + \frac{C_\eta}{\Delta t} ||\bm{\eta}^n_h||_E^2\Big].
    \end{split}
\end{align}

\noindent \textit{\textbf{Step 2: Bound the elements in $Z^h$.}}\\

Next, we bound the elements $p_f^{n+1}$ and $g_{1,h}^{n+1}$ using the inf-sup condition in Theorem \ref{FPSI_FEM:thm:DiscInfSup}.
First, find an expression for $b_{MZ}(\bm{\zeta}_h,(p_{f,h}^{n+1},g_{1,h}^{n+1}))$ by returning to the weak form \eqref{FPSI_FEM:WF:allTimeDisc_StabError_spaceDisc}, choosing the test functions $q_h = 0$ and $s_{1,h} = 0$  to obtain for all $\bm{\zeta}_h = (\bm{v}_h, \bm{\varphi}_h, w_h, s_{2,h}, \mu_h) \in M^h$:

\begin{align}\label{FPSI_FEM:eq:testfuncZ0}
    \begin{split}
       & b_{MZ}\left(\bm{\zeta}_h,(p_{f,h}^{n+1},g_{1,h}^{n+1}) \right) =-  (p_{f,h}^{n+1}, \nabla \cdot \bm{v}_h)_{\Omega_f} -  \langle g_{1,h}^{n+1} \bm{n}_f,\bm{v}_h\rangle_\gamma - \langle g_{1,h}^{n+1}\bm{n}_p, \bm{\varphi}_h \rangle_\gamma + \langle g_{1,h}^{n+1},\mu_h \rangle_\gamma  \\
      &\hspace{5mm} =  \langle\bm{f}_f^{n+1},\bm{v}_h\rangle_{\Omega_f} - \rho_f ( \bm{\dot{u}}_h^{n+1}, \bm{v}_h )_{\Omega_f} - 2  \nu_f \left( D(\bm{u}_h^{n+1}), D(\bm{v}_h) \right)_{\Omega_f}  +  ( g_{2,h}^{n+1} \bm{\tau}_\gamma,\bm{v}_h)_\gamma \\
&\hspace{5mm}+  \langle\bm{f}_\eta^{n+1}, \bm{\varphi}_h \rangle_{\Omega_p} - \frac{\rho_p}{\Delta t} \left( \bm{\dot{\eta}}_h^{n+1} - \bm{\dot{\eta}}_h^n, \bm{\varphi}_h \right)_{\Omega_p} -  2  \nu_p \left( D\left(\bm{\eta}_h^{n+1}\right), D(\bm{\varphi}_h) \right)_{\Omega_p} \\
& \hspace{5mm}- \lambda  \left( \nabla \cdot  \bm{\eta}^{n+1}_h, \nabla \cdot \bm{\varphi}_h \right)_{\Omega_p}  + \alpha (p_{p,h}^{n+1}, \nabla \cdot \bm{\varphi}_h)_{\Omega_p} -  ( g_{2,h}^{n+1}\bm{\tau}_\gamma, \bm{\varphi}_h)_\gamma   + \langle f_p^{n+1},w \rangle_{\Omega_p} 
\\
&\hspace{5mm}- s_0 ( \bm{\dot}{p}_{p,h}^{n+1}, w_h)_{\Omega_p} - \alpha\left(\nabla \cdot \bm{\dot{\eta}}_h^{n+1},w_h \right)_{\Omega_p} - \kappa  (\nabla p_{p,h}^{n+1}, \nabla w_h)_{\Omega_p} +  ( \lambda_{p,h}^{n+1}, w)_{\gamma}  - (  p_{p,h}^{n+1},\mu_h)_\gamma \\
      & \hspace{5mm}  - \overline{\epsilon} (\lambda_{p,h}^{n+1},\mu_h)_{1/2,\gamma}  -  \frac{1}{\beta } ( g_{2,h}^{n+1},s_{2,h})_\gamma  -  (\bm{u}_h^{n+1} \cdot \bm{\tau}_\gamma,s_{2,h})_\gamma + \left(\bm{\dot{\eta}}_h^{n+1} \cdot \bm{\tau}_\gamma,s_{2,h}\right)_\gamma .
    \end{split}
\end{align}

\noindent After using the Cauchy-Schwarz inequality and the trace inequalities \eqref{FPSI_FEM:eq:TraceIneqs}, we factor the test functions out to obtain the upper bound
\begin{align*}
b_{MZ}&\left(\bm{\zeta}_h,(p_{f,h}^{n+1},g_{1,h}^{n+1}) \right) 
 \leq \Big( ||\bm{v}_h||_{1} + ||\bm{\varphi}||_{1} + ||w_h||_{1} + ||\mu_h||_{1/2,\gamma}+ ||s_{2,h}||_{0,\gamma} \Big) \\
 &\Bigg[ ||\bm{f}_f^{n+1}||_{-1} + \rho_f ||\bm{\dot{u}}_h^{n+1}||_{0}  + 2\nu_f ||D(\bm{u}_h^{n+1})||_{0} +   C_T||g_{2,h}^{n+1}||_{0,\gamma}    +  ||\bm{f}_\eta^{n+1}||_{-1}  \\
 & + \frac{\rho_p}{\Delta t} ||\bm{\dot{\eta}}_h^{n+1} - \bm{\dot{\eta}}_h^n||_{0}+ 2\nu_p ||D(\bm{\eta}_h^{n+1})||_{0} +\lambda \sqrt{d}C_K  ||\nabla \cdot \bm{\eta}_h^{n+1} ||_{0} + \alpha \sqrt{d}C_K ||p_{p,h}^{n+1}||_{0} \\
 & + C_T||g_{2,h}^{n+1}||_{0,\gamma}   +  ||f_p^{n+1}||_{-1} + s_0 ||\bm{\dot}{p}_{p,h}^{n+1}||_{0} + \alpha ||\nabla \cdot \bm{\dot{\eta}}^{n+1}||_{0} + \kappa ||\nabla p_{p,h}^{n+1}||_{0} \\
 &+ C_T||\lambda_{p,h}^{n+1}||_{0,\gamma}  + ||p_{p,h}^{n+1}||_{0,\gamma} + \overline{\epsilon} ||\lambda_{p,h}^{n+1}||_{1/2,\gamma} + \frac{1}{\beta} ||g_{2,h}^{n+1}||_{0,\gamma} + ||\bm{u}_h^{n+1}||_{0,\gamma} + ||\bm{\dot{\eta}}_h^{n+1}||_{0,\gamma} \Bigg] \\
    &\leq \sqrt{5} \Big( ||\bm{v}_h||^2_{1} + ||\bm{\varphi}||^2_{1} + ||w_h||^2_{1} + ||\mu_h||^2_{1/2,\gamma}+ ||s_{2,h}||^2_{0,\gamma} \Big)^{1/2} \mathcal{J}_1 ,
\end{align*}
where $\mathcal{J}_1$ denotes the second factor in the inequality in brackets. Plugging this result into the inf-sup condition in Theorem \ref{FPSI_FEM:thm:DiscInfSup} gives
\begin{align}\label{FPSI_FEM:J2bound}
     &\beta_2 || (p_{f,h}^{n+1},g_{1,h}^{n+1})||_Z \leq  \underset{0 \neq \bm{\zeta}_h \in M^h}{\sup} \frac{b_{MZ}\left(\bm{\zeta}_h,(p_{f,h}^{n+1},g_{1,h}^{n+1}) \right) }{||\bm{\zeta}_h||_{M}} \leq  \underset{0 \neq \bm{\zeta}_h \in M^h}{\sup} \frac{\sqrt{5}||\bm{\zeta}_h||_M \mathcal{J}_1}{||\bm{\zeta}_h||_{M}} \leq \sqrt{5} \ \mathcal{J}_1.
\end{align}

\noindent Squaring both sides of \eqref{FPSI_FEM:J2bound}, apply trace and Poincar\'{e} inequalities \eqref{FPSI_FEM:eq:TraceIneqs}-\eqref{FPSI_FEM:eq:PoincareIneqs} again with $C_1, C_2$ defined in \eqref{FPSI_FEM:eq:StabConstants}. Multiply the resulting inequality by $\varepsilon_1 \Delta t^2$, with $\varepsilon_1>0$ the constant from Young's inequality in \eqref{eq:IBP_inTime_Youngs}. Lastly, sum from $n=0$ to $n=N-1$, simplifying to obtain:
\begin{align}\label{FPSI_FEM:eq:stabboundpg1}
\begin{split}
    \varepsilon_1 &\Delta t^2 \beta_2^2 \sum_{n=0}^{N-1}\left(||p_{f,h}^{n+1}||_{0}^2 + ||g_{1,h}^{n+1}||_{-1/2,\gamma}^2 \right) \leq  70\varepsilon_1 \Delta t^2 \sum_{n=0}^{N-1} \Big[ ||\bm{f}_f^{n+1}||^2_{-1} +   ||\bm{f}_\eta^{n+1}||^2_{-1} +  ||f_p^{n+1}||^2_{-1} \\
    &+ \rho_f^2 ||\bm{\dot{u}}_h^{n+1}||^2_{0} + (2\nu_f + C_1)^2 ||D(\bm{u}_h^{n+1})||^2_{0}  +  K_2||\bm{\dot{\eta}}_h^{n+1}||_E^2  + \frac{\rho_p^2}{\Delta t^2} ||\bm{\dot{\eta}}_h^{n+1} - \bm{\dot{\eta}}_h^n||^2_{0}  +K_1 ||\bm{\eta}_h^{n+1}||^2_{E}  \\
    &+ s_0^2 ||\bm{\dot}{p}_{p,h}^{n+1}||^2_{0}  + (\kappa + C_1 + \alpha C_2)^2 ||\nabla p_{p,h}^{n+1}||^2_{0}  + (C_T + \overline{\epsilon})^2 ||\lambda_{p,h}^{n+1}||^2_{1/2,\gamma} + \left(2C_T + \frac{1}{\beta}\right)^2 ||g_{2,h}^{n+1}||_{0,\gamma}^2 \Big],
    \end{split}
\end{align}
where $K_1 := \max\{2\nu_p, d \lambda C_K^2 \}$ and $K_2 := \max\left\{ \dfrac{C_1^2}{2\nu_p},\dfrac{\alpha^2}{\lambda}\right\}$.

\noindent \textit{\textbf{Step 3: Combine bounds and apply Gronwall's lemma.}}\\

Next, sum the inequalities \eqref{FPSI_FEM:stabboundMterms2} and \eqref{FPSI_FEM:eq:stabboundpg1}, absorbing appropriate terms from \eqref{FPSI_FEM:eq:stabboundpg1} on the left hand side of the inequality:
\begin{align}\label{FPSI_FEM:eq:AlmostToGronwall}
    \begin{split}
       &\rho_f ||  \bm{u}_h^{N} ||^2_{0}      + \rho_p || \bm{\dot{\eta}}_h^{N}||^2_{0}    +s_0  || p_{p,h}^{N}||^2_{0}  + || \bm{\eta}_h^{N} ||_E^2 + \Delta t \sum_{n=0}^{N-1} \Big[ \rho_p \Delta t (1-70\varepsilon_1 \rho_p) || \bm{\ddot{\eta}}_h^{n+1} ||^2_{0} \\
       &\hspace{5mm}+  \rho_f \Delta t(1 - 70\varepsilon_1 \rho_f) ||  \bm{\dot{u}}_h^{n+1} ||^2_{0}  + s_0 \Delta t(1 - 70\varepsilon_1 s_0) || \bm{\dot}{p}_{p,h}^{n+1}||^2_{0} + \Delta t(1-70\varepsilon_1 K_2) || \bm{\dot{\eta}}_h^{n+1} ||_E^2  \\
       &\hspace{5mm}+   (2 \nu_f-70\varepsilon_1 \Delta t (2\nu_f + C_1)^2) || D( \bm{u}_h^{n+1})||^2_{0} + (\kappa-70\varepsilon_1 \Delta t (\kappa + C_1 + \alpha C_2)^2)  || \nabla p_{p,h}^{n+1}||^2_{0}  \\
       &\hspace{5mm}+ \left( \frac{2}{\beta } - 70\varepsilon_1 \Delta t(2C_T+\frac{1}{\beta})^2 \right) || g_{2,h}^{n+1}||^2_{0,\gamma} + (2\overline{\epsilon}-70\varepsilon_1 \Delta t (C_T+\bar{\epsilon})^2) ||\lambda_{p,h}^{n+1}||^2_{1/2,\gamma} \\
       &\hspace{5mm}+    \varepsilon_1 \Delta t \beta_2^2 \left(||p_{f,h}^{n+1}||_{0}^2 + ||g_{1,h}^{n+1}||_{-1/2,\gamma}^2 \right) \Big] \\
      & \leq   \rho_f  ||  \bm{u}_h^{0} ||^2_{0} +  \rho_p || \bm{\dot{\eta}}_h^{0}||^2_{0} + s_0  || p_{p,h}^{0}||^2_{0} +   (1+C_\eta) ||\bm{\eta}_h^0||_E^2 +  || \bm{f}_\eta^0||_{-1}^2 +  \frac{1}{\Delta t^2 \varepsilon_1} || \bm{f}_\eta^N||_{-1}^2 \\
      & \hspace{5mm} +  \Delta t^2 \varepsilon_1 C_\eta ||\bm{\eta}^N_h ||_E^2 + \Delta t \sum_{n=0}^{N-1} \Big[\left( \frac{C_{KP}}{2\nu_f} + 70\varepsilon_1 \Delta t \right) ||\bm{f}_f^{n+1}||_{-1}^2    +  \left( \frac{C_{KP}}{\kappa} + 70\varepsilon_1 \Delta t \right)|| f_p^{n+1}||^2_{-1}  \\
      &  + \Delta t || \bm{\dot{f}_\eta}^{n+1} ||_{-1}^2  \hspace{5mm} + 70\varepsilon_1 \Delta t  ||\bm{f}_\eta^{n+1}||^2_{-1} + \frac{C_\eta}{\Delta t} ||\bm{\eta}^n_h||_E^2+ 70\varepsilon_1 \Delta t K_1 ||\bm{\eta}_h^{n+1}||^2_{E} \Big].
    \end{split}
\end{align}

For remaining terms involving $||\bm{\eta}_h||_E^2$ on the right hand side, we apply the discrete Gronwall's lemma as stated in (\cite{Layton_2008}), making some adjustments so the inequality is in the correct form. Defining the constant $d^n:=\left( \frac{C_\eta}{\Delta t} + 70\varepsilon_1 \Delta t K_1\right) $, the last two terms of the sum on the right hand side of \eqref{FPSI_FEM:eq:AlmostToGronwall} can be expressed as 
\begin{align*}
    &\Delta t \sum_{n=0}^{N-1} \Big[\frac{C_\eta}{\Delta t} ||\bm{\eta}^n_h||_E^2+ 70\varepsilon_1 \Delta t K_1 ||\bm{\eta}_h^{n+1}||^2_{E}\Big]= C_\eta ||\bm{\eta}_h^0||_E^2 + \Delta t^2 70 \varepsilon_1 K_1 ||\bm{\eta}_h^N||_E^2 + \Delta t \sum_{n=0}^{N-2} d^n ||\bm{\eta}_h^{n+1}||^2_{E}.
\end{align*}

\noindent Let the constants $B_0, B_1, B_2 \geq 0$ and non-negative sequences $\bar{b}^{n},b^n,$ and $ c^{n}$ for $n=0,\ldots,N$ be defined as
\begin{align*}
    B_0&:= \rho_f ||  \bm{u}_h^{N} ||^2_{0}      + \rho_p || \bm{\dot{\eta}}_h^{N}||^2_{0}    +s_0  || p_{p,h}^{N}||^2_{0}\\
   B_1&:= \rho_f  ||  \bm{u}_h^{0} ||^2_{0} +  \rho_p || \bm{\dot{\eta}}_h^{0}||^2_{0} + s_0  || p_{p,h}^{0}||^2_{0} +   (1+2C_\eta) ||\bm{\eta}_h^0||_E^2  +  || \bm{f}_\eta^0||_{-1}^2 +  \frac{1}{\Delta t^2 \varepsilon_1} || \bm{f}_\eta^N||_{-1}^2,\\
    B_2 &:=\rho_p \Delta t ||\bm{\ddot{\eta}}_h^0||_0^2 + \rho_f \Delta t ||\bm{\dot{u}}_h^0||_0^2 + s_0 \Delta t ||\bm{\dot}{p}_{p,h}^0||_0^2 + \Delta t ||\bm{\dot{\eta}}_h^0||_E^2 + 2\nu_f ||D(\bm{u}_h^0)||_0^2 \\
    &+ \kappa ||\nabla p_{p,h}^0||_0^2 + \frac{2}{\beta} ||g_{2,h}^0||^2_{0,\gamma} + 2\bar{\epsilon} ||\lambda_{p,h}^0||_{1/2,\gamma}^2 + \varepsilon_1 \Delta t \beta_2^2 (||p_{f,h}^0||_0^2 + ||g_{1,h}^0||_{-1/2,\gamma}^2)\\ 
   \overline{b}^{n} &:=  \Big[ \rho_p \Delta t (1-70\varepsilon_1 \rho_p)  || \bm{\ddot{\eta}}_h^{n}||_0^2 +  \rho_f \Delta t(1 - 70\varepsilon_1 \rho_f) ||  \bm{\dot{u}}_h^{n} ||^2_{0}  + s_0 \Delta t(1 - 70\varepsilon_1 s_0) || \bm{\dot}{p}_{p,h}^{n}||^2_{0} \\
       &+ \Delta t(1-70\varepsilon_1 K_2) || \bm{\dot{\eta}}_h^{n} ||_E^2 +   (2 \nu_f-70\varepsilon_1 \Delta t (2\nu_f + C_1)^2) || D( \bm{u}_h^{n})||^2_{0}  \\
       &+ (\kappa-70\varepsilon_1 \Delta t (\kappa + C_1 + \alpha C_2)^2)  || \nabla p_{p,h}^{n}||^2_{0} + \left( \frac{2}{\beta } - 70\varepsilon_1 \Delta t(2C_T+\frac{1}{\beta})^2 \right) || g_{2,h}^{n}||^2_{0,\gamma}  \\
       &+ (2\overline{\epsilon}-70\varepsilon_1 \Delta t (C_T+\bar{\epsilon})^2) ||\lambda_{p,h}^{n}||^2_{1/2,\gamma} +    \varepsilon_1 \Delta t \beta_2^2 \left(||p_{f,h}^{n}||_{0}^2 + ||g_{1,h}^{n}||_{-1/2,\gamma}^2 \right) \Big] \\
   c^{n} &:=  \left( \frac{C_{KP}}{2\nu_f} + 70\varepsilon_1 \Delta t \right) ||\bm{f}_f^{n}||_{-1}^2    +  \left( \frac{C_{KP}}{\kappa} + 70\varepsilon_1 \Delta t \right)|| f_p^{n}||^2_{-1}    + \Delta t || \bm{\dot{f}_\eta}^{n} ||_{-1}^2 + 70\varepsilon_1 \Delta t  ||\bm{f}_\eta^{n}||^2_{-1}\\
    b^n &:= \frac{1}{N \Delta t} B_0 + \bar{b}^n.
\end{align*}
Thus, \eqref{FPSI_FEM:eq:AlmostToGronwall} is equivalent to the inequality
\begin{align}\label{FPSI_FEM:eq:ApplyGronwallToThis}
    \begin{split}
       & (1-  \Delta t^2( \varepsilon_1 C_\eta +70 \varepsilon_1 K_1))||\bm{\eta}^N_h ||_E^2 + \Delta t \sum_{n=1}^{N} b^{n} \leq   \Delta t \sum_{n=0}^{N-1} \Big[d^{n}  ||\bm{\eta}_h^{n}||^2_{E}\Big] +\Delta t \sum_{n=0}^{N} c^{n}  +B_1,
    \end{split}
\end{align}
where we have added the terms $\Delta t d^0 ||\bm{\eta}^0_h||_E^2 + \Delta t c_0$ to the right hand side as an upper bound.
As we also need a bound for the initial term $\Delta t b^0$, note
that  \eqref{FPSI_FEM:eq:ApplyGronwallToThis} implies that $B_0$ is less than the right hand side and $\bar{b}^0 \leq B_2$,
\begin{align*} 
\Delta t b^0 &= \frac{1}{N} B_0 +\bar{b}^0 \leq \frac{1}{N} \Bigg(\Delta t \sum_{n=0}^{N-1} \Big[d^{n}  ||\bm{\eta}_h^{n}||^2_{E}\Big] +\Delta t \sum_{n=0}^{N} c^{n}  +B_1\Bigg) + \Delta t B_2 .
\end{align*} 
Let $K_3 := 1+\frac{1}{N}$. Thus, adding the term $\Delta t b_0$
to the left of \eqref{FPSI_FEM:eq:ApplyGronwallToThis} yields
\begin{align}\label{FPSI_FEM:eq:ApplyGronwallToThisFinal}
    \begin{split}
        &(1-  \Delta t^2( \varepsilon_1 C_\eta +70 \varepsilon_1 K_1))||\bm{\eta}^N_h ||_E^2 + \Delta t \sum_{n=0}^{N} b^{n}\\
        &\leq  K_3 \Delta t \sum_{n=0}^{N-1} \Big[d^{n}  ||\bm{\eta}_h^{n}||^2_{E}\Big] +K_3\Delta t \sum_{n=0}^{N} c^{n}  + K_3B_1 + \Delta t B_2.
    \end{split}
\end{align}

\noindent To ensure positivity of the terms on the left hand side of \eqref{FPSI_FEM:eq:ApplyGronwallToThisFinal}, we pick $\varepsilon_1>0$ as 
\begin{align}\label{FPSI_FEM:eq:eps1defn}
\begin{split}
    \varepsilon_1 := \frac{1}{2}\min\Big\{& \frac{1}{70\rho_p}, \frac{1}{70\rho_f},\frac{1}{70s_0}, \frac{1}{70 K_2}, \frac{2\nu_f}{70\Delta t (2\nu_f+C_1)^2}, \frac{\kappa}{70\Delta t (\kappa + C_1 + \alpha C_2)^2},\\
    &\frac{2}{70\beta \Delta t (2C_T + 1/\beta)^2}, \frac{2\bar{\epsilon}}{70\Delta t(C_T + \bar{\epsilon})^2}, \frac{1}{\Delta t^2 (C_\eta + 70K_1)} \Big\}.
    \end{split}
\end{align}


\noindent Define $\varepsilon_2 := (1-  \Delta t^2(\varepsilon_1 C_\eta +70 \varepsilon_1 K_1))$, which is positive by the definition of $\varepsilon_1$. Identifying the term $a^{n+1} := ||\bm{\eta}_h^{n+1}||_E^2$, Gronwall's lemma (\cite{Layton_2008}) applied to \eqref{FPSI_FEM:eq:ApplyGronwallToThisFinal} results in
\begin{align}\label{FPSI_FEM:eq:FinalPostGronwall}
    \begin{split}
       &\varepsilon_2 ||\bm{\eta}^N_h ||_E^2 + \Delta t \sum_{n=0}^{N} b^{n} \leq  \text{exp}\Bigg( \frac{N\Delta t}{\varepsilon_2}K_3d^{n}  \Bigg) \Bigg(K_3\Delta t \sum_{n=0}^{N}  c^{n}  +K_3B_1 + \Delta t B_2\Bigg).
    \end{split}
\end{align}

\noindent To simplify the constants multiplying the norms on the left hand side and the exponential term on the right hand side, define $\overline{C}$ and $C^*$ as follows, and the desired estimate is shown:
\begin{align}\label{FPSI_FEM:eq:Cfinaldefns}
\begin{split}
    \overline{C} &:= \min \Big\{ \varepsilon_2, K_3\rho_f, K_3\rho_p, K_3s_0, \rho_p \Delta t (1-70\varepsilon_1 \rho_p), \rho_f \Delta t (1-70\varepsilon_1 \rho_f), s_0 \Delta t (1-70\varepsilon_1 s_0),\\
    & \Delta t (1-70\varepsilon_1 K_2),2\nu_f - 70\varepsilon_1 \Delta t (2\nu_f + C_1)^2, \kappa - 70\varepsilon_1 \Delta t (\kappa + C_1 + \alpha C_2)^2, \\
    & \frac{2}{\beta} - 70\varepsilon_1 \Delta t (2C_T + \frac{1}{\beta})^2, 2\bar{\epsilon}-70\varepsilon_1 \Delta t (C_T + \bar{\epsilon})^2, \varepsilon_1 \Delta t \beta_2^2 \Big\}, \\
    C^* &:= \text{exp}\Bigg( \dfrac{N\Delta t}{\varepsilon_2} K_3d^{n}  \Bigg) = \text{exp}\Bigg( \left(\dfrac{N}{\varepsilon_2} +\dfrac{1}{\varepsilon_2}\right)(C_\eta + 70 \varepsilon_1 \Delta t^2 K_1) \Bigg).
    \end{split}
\end{align}

\end{proof}

	\section{Conclusions}\label{FPSI_joint:sec:Conclusion}

We have proposed a monolithic formulation for the a fluid-poroelastic interaction system involving the fully dynamic, two-field Biot model and the dynamic Stokes equations. Although posed as a monolithic problem, we have chosen the Lagrange multipliers in such a way to create a saddle point system which is favorable for domain decomposition. In this paper, we demonstrated the well-posedness and stability of both the semi-discrete and fully discrete formulations, adding a small stabilization term for the well-posedness of the semi-discrete system. The goal of this formulation is to facilitate domain decomposition. Future work will perform error analysis on the monolithic formulation in order to evaluate the domain decomposition method and develop a partitioned method for this saddle point system. Solving a Schur complement equation for particular variables will enable the decoupling of the fluid and poroelastic subdomains, allowing them to be independently updated in parallel at each time step.

\small 
\bibliography{DissertationBib_FinalCombined}

@article{Wang_2025,
  title={A {L}agrange multiplier formulation for the fully dynamic {N}avier--{S}tokes--{B}iot system},
  author={X. Wang and I. Yotov},
  journal={IMA Journal of Numerical Analysis},
  pages={draf074},
  year={2025},
  publisher={Oxford University Press}
}

@book{Boffi_2013,
title = "Mixed {F}inite {E}lement {M}ethods and {A}pplications",
author = "D. Boffi and F. Brezzi and M. Fortin",
publisher = "Springer",
volume="44",
year = "2013"
}

@book{Layton_2008,
  title={Introduction to the {N}umerical {A}nalysis of {I}ncompressible {V}iscous {F}lows},
  author={W. Layton},
  year={2008},
  publisher={SIAM}
}

@article{Brezzi_1990,
  title={A discourse on the stability conditions for mixed finite element formulations},
  author={F. Brezzi and K.-J. Bathe},
  journal={Computer Methods in Applied Mechanics and Engineering},
  volume={82},
  number={1-3},
  pages={27--57},
  year={1990},
  publisher={Elsevier}
}

@article{Murad_2001,
  title={Micromechanical computational modeling of secondary consolidation and hereditary creep in soils},
  author={M.A. Murad and J.N. Guerreiro and A.F.D. Loula},
  journal={Computer Methods in Applied Mechanics and Engineering},
  volume={190},
  number={15-17},
  pages={1985--2016},
  year={2001},
  publisher={Elsevier}
}

@incollection{Detournay_1993,
  title={Fundamentals of poroelasticity},
  author={E. Detournay and A.H.-D. Cheng},
  booktitle={Analysis and Design Methods},
  pages={113--171},
  year={1993},
  publisher={Elsevier}
}

@incollection{Showalter_2005,
  title={Poroelastic filtration coupled to {S}tokes flow},
  author={R.E. Showalter},
  booktitle={Control Theory of Partial Differential Equations},
  pages={243--256},
  year={2005},
  publisher={Chapman and Hall/CRC}
}

@article{Ambartsumyan_2018,
  title={A {L}agrange multiplier method for a {S}tokes--{B}iot fluid--poroelastic structure interaction model},
  author={Ambartsumyan, I. and Khattatov, E. and Yotov, I. and Zunino, P.},
  journal={Numerische Mathematik},
  volume={140},
  number={2},
  pages={513--553},
  year={2018},
  publisher={Springer}
}

@article{Bociu_2021,
    author = "L. Bociu and S. Canic and B. Muha and J.T. Webster",
    title = "Multilayered Poroelasticity Interacting with {S}tokes Flow",
    journal = "SIAM Journal on Mathematical Analysis",
    volume= "53",
    number="6",
    pages="6243--6279",
    year="2021",
    publisher="SIAM"}

@article{Caucao_2022,
author = "S. Caucao and T. Li and I. Yotov",
title = "A multipoint stress--flux mixed finite element method for the {S}tokes--{B}iot model",
journal = "Numerische Mathematik",
volume = "152",
year = "2022",
pages = "411--473"
}

@article{Avalos_2024,
  title="Weak and strong solutions for a fluid--poroelastic--structure interaction via a semigroup approach",
  author="G. Avalos and E. Gurvich and J.T. Webster",
  journal="Mathematical Methods in the Applied Sciences",
  year="2024",
  publisher="Wiley Online Library"
}

@article{Kumar_2020,
  title={Conservative discontinuous finite volume and mixed schemes for a new four--field formulation in poroelasticity},
  author={S. Kumar and R. Oyarz{\'u}a and R. Ruiz-Baier and R. Sandilya},
  journal={ESAIM: Mathematical Modelling and Numerical Analysis},
  volume={54},
  number={1},
  pages={273--299},
  year={2020},
  publisher={EDP Sciences}
}

@article{Cesmelioglu_2020,
  title={Numerical analysis of the coupling of free fluid with a poroelastic material},
  author={A. Cesmelioglu and P. Chidyagwai},
  journal={Numerical Methods for Partial Differential Equations},
  volume={36},
  number={3},
  pages={463--494},
  year={2020},
  publisher={Wiley Online Library}
}

@article{Bukac_2015OpSplit,
  title="An operator splitting approach for the interaction between a fluid and a multilayered poroelastic structure",
  author="M. Buka{\v{c}} and I. Yotov and P. Zunino",
  journal="Numerical Methods for Partial Differential Equations",
  volume="31",
  number="4",
  pages="1054--1100",
  year="2015",
  publisher="Wiley Online Library"
}

@article{Causin_2014,
  title={A poroelastic model for the perfusion of the lamina cribrosa in the optic nerve head},
  author={P. Causin and G. Guidoboni and A. Harris and D. Prada and R. Sacco and S. Terragni},
  journal={Mathematical Biosciences},
  volume={257},
  pages={33--41},
  year={2014},
  publisher={Elsevier}
}

@article{Banks_2017,
  title={Sensitivity analysis in poro--elastic and poro--visco--elastic models with respect to boundary data},
  author={H.T. Banks and K. Bekele-Maxwell and L. Bociu and M. Noorman and G. Guidoboni},
  journal={Quarterly of Applied Mathematics},
  volume={75},
  number={4},
  pages={697--735},
  year={2017},
  publisher={JSTOR}
}

@article{Calo_2008,
  title={Multiphysics model for blood flow and drug transport with application to patient--specific coronary artery flow},
  author={V.M. Calo and N.F. Brasher and Y. Bazilevs and T.J.R. Hughes},
  journal={Computational Mechanics},
  volume={43},
  pages={161--177},
  year={2008},
  publisher={Springer}
}

@misc{Chen_2021,
title = "Inf--sup conditions for operator equations",
author = "L. Chen",
year = "2021",
howpublished = "\url{https://www.math.uci.edu/\~ chenlong/226/infsupOperator.pdf}",
note = "Last accessed: 2025-01-01"
}

@misc{Chen_Sobolev,
title = "Sobolev spaces and elliptic equations",
author = "L. Chen",
year = "2022",
howpublished = "\url{https://www.math.uci.edu/~chenlong/226/Ch1Space.pdf}",
note = "Last accessed: 2025-01-03"
}

@article{Gatica_2011_structure,
  title="Analysis of fully--mixed finite element methods for the {S}tokes-{D}arcy coupled problem",
  author="G. Gatica and R. Oyarz{\'u}a and F.-J. Sayas",
  journal="Mathematics of Computation",
  volume="80",
  number="276",
  pages="1911--1948",
  year="2011"
}

@article{Ambartsumyan_2019,
  title={A nonlinear {S}tokes--{B}iot model for the interaction of a non-{N}ewtonian fluid with poroelastic media},
  author={I. Ambartsumyan and V.J. Ervin and T. Nguyen and I. Yotov},
  journal={ESAIM: Mathematical Modelling and Numerical Analysis},
  volume={53},
  number={6},
  pages={1915--1955},
  year={2019},
  publisher={EDP Sciences}
}

@article{Cesmelioglu_2017,
  title={Analysis of the coupled {N}avier--{S}tokes/{B}iot problem},
  author={A. Cesmelioglu},
  journal={Journal of Mathematical Analysis and Applications},
  volume={456},
  number={2},
  pages={970--991},
  year={2017},
  publisher={Elsevier}
}

@article{Lee_2023,
  title={Locking--free and locally--conservative enriched {G}alerkin method for poroelasticity},
  author={S. Lee and S.-Y. Yi},
  journal={Journal of Scientific Computing},
  volume={94},
  number={1},
  pages={26},
  year={2023},
  publisher={Springer}
}

@book{Lions_2012,
  title={Non--homogeneous boundary value problems and applications: {V}ol. 1},
  author={J.L. Lions and E. Magenes},
  volume={181},
  year={2012},
  publisher={Springer Science \& Business Media}
}

@article{Li_2022_Hydro,
  title={A mixed elasticity formulation for fluid--poroelastic structure interaction},
  author={T. Li and I. Yotov},
  journal={ESAIM: Mathematical Modelling and Numerical Analysis},
  volume={56},
  number={1},
  pages={1--40},
  year={2022},
  publisher={EDP Sciences}
}

@article{deCastro_CAMWA_2025,
author = "A. de Castro and H. Lee and M.M. Wiecek",
title = "A {L}agrange multiplier method for fluid--structure interaction: {W}ell--posedness and domain decomposition",
journal = "Computers \& Mathematics with Applications",
year = "2025",
volume = "181",
pages = "193--215"
}
    \bibliographystyle{plainurl}

\end{document}